\documentclass[hidelinks,onefignum,onetabnum]{siamart220329}

\usepackage[utf8]{inputenc}
\usepackage{amsfonts,amssymb}
\usepackage[margin = 1.3in]{geometry}
\usepackage{hyperref}
\usepackage{cleveref}
\usepackage[normalem]{ulem}
\usepackage{subfig}
\usepackage{dsfont}
\usepackage{xcolor}
\usepackage{algorithm}
\usepackage{tcolorbox}
\usepackage{algorithmicx}
\usepackage{algpseudocode}

\newcommand{\R}{\mathbb{R}}
\newcommand{\F}{\mathcal{F}}
\newcommand{\rd}{\mathrm{d}}
\newcommand{\f}{\widetilde{f}}
\newcommand{\g}{\widetilde{u}}
\newcommand{\feta}{\widetilde{f\,}^\eta}
\newcommand{\geta}{\widetilde{u\,}^\eta}
\newsiamremark{remark}{Remark}

\makeatletter
\newcommand{\refcheckize}[1]{%
  \expandafter\let\csname @@\string#1\endcsname#1%
  \expandafter\DeclareRobustCommand\csname relax\string#1\endcsname[1]{%
    \csname @@\string#1\endcsname{##1}\wrtusdrf{##1}}%
  \expandafter\let\expandafter#1\csname relax\string#1\endcsname
}
\makeatother

\ifpdf
  \DeclareGraphicsExtensions{.eps,.pdf,.png,.jpg}
\else
  \DeclareGraphicsExtensions{.eps}
\fi

\headers{Multi-window STFT phase retrieval}{Rima Alaifari and Yunan Yang}

\title{Multi-window approaches for direct and stable STFT Phase Retrieval\thanks{Submitted to the editors.
\funding{R.~Alaifari acknowledges support through SNSF Grant 200021$\_$184698. Y.~Yang acknowledges support from NSF through grant DMS-2409855, Office of Naval Research  through grant
N00014-24-1-2088, and Cornell PCCW Affinito-Stewart Grant. Part of the research was performed when Y.~Yang was an Advanced Fellow at the Institute for Theoretical Studies, ETH Z\"urich, and received support from Dr.~Max R\"ossler, the Walter Haefner Foundation, and the ETH
Z\"urich Foundation.}}}

\author{Rima Alaifari\thanks{Seminar of Applied Mathematics, ETH Z\"urich, Z\"urich, Switzerland (\email{rima.alaifari@math.ethz.ch}).}
\and  Yunan Yang\thanks{Department of Mathematics, Cornell University, Ithaca, NY, USA (\email{yunan.yang@cornell.edu}).}
}

\ifpdf
\hypersetup{
  pdftitle={Multi-window STFT Phase Retrieval},
  pdfauthor={Rima Alaifari and Yunan Yang}
}
\fi

\begin{document}

\maketitle

\begin{abstract}
Phase retrieval from phaseless short-time Fourier transform (STFT) measurements is known to be inherently unstable when measurements are taken with respect to a single window. While an explicit inversion formula exists, it is useless in practice due to its instability. In this paper, we overcome this lack of stability by presenting two multi-window approaches that rely on a ``good coverage" of the time-frequency plane by the ambiguity functions of the windows. The first is to use the fractional Fourier transform of a \textit{dilated} Gauss function with various angles as window functions. The essential support of a superposition of the ambiguity function from such window functions is of a ``daffodil shape", which converges to a large disc as more angles are used, yielding a much broader coverage in the time-frequency domain. The second approach uses Hermite functions of various degrees as the window functions. The larger the degree, the wider the ambiguity function but with zeros on circles in the time-frequency domain. Combining Hermite functions of different degrees, we can achieve a wide coverage with zeros compensated by the essential support of the ambiguity function from other Hermite windows. Taking advantage of these multi-window procedures, we can stably perform STFT phase retrieval using the direct inversion formula.
\end{abstract}



\section{Introduction}
In many scientific and engineering fields, seeking to reconstruct a signal or image from magnitude measurements alone is a common challenge. In these applications, only the absolute values of a set of linear measurements can be acquired so that phase information is missing. In some instances, the phase cannot be directly measured or the phase measurements are too noisy to be of use. \emph{Phase retrieval} addresses this problem by using computational algorithms that estimate the missing phase data, enabling the full reconstruction of the original signal or object. This technique is essential in areas such as optics, crystallography, and diffraction imaging, where direct phase measurement is often impractical. Specific audio processing applications also require phase retrieval from certain magnitude-only measurements. 

In this paper, we focus on a specific instance of phase retrieval, which aims at recovery from magnitude measurements of the \emph{short-time Fourier transform (STFT),} that we will first introduce: For a given \emph{window function} $g \in L^2(\R)$, the linear mapping that takes any function $f \in L^2(\R)$ to its short-time Fourier transform (STFT) or \emph{windowed Fourier transform} is defined as
\begin{equation}\label{eq:gabor}
    V_g f(x,y):= \int_{\R} f(t) \overline{g(t-x)} e^{-2\pi i t y} \, \rd t.
\end{equation}
The special case of the STFT with a Gauss window $g(t)=\varphi(t) :=  e^{-\pi t^2}$ is sometimes referred to as the \emph{Gabor transform}. The STFT enjoys many nice properties such as the fact that it is a linear isometry from $L^2(\R)$ to $L^2(\R^2)$ up to the factor $\|g\|_{L^2(\R)},$ i.e.,
$$
\|V_g f\|_{L^2(\R^2)} = \|f\|_{L^2(\R)} \|g\|_{L^2(\R)}.
$$
Thus, the reconstruction from the STFT is unique and stable, since for any $f_1, f_2 \in L^2(\R)$:
$$
\|V_g f_1-V_g f_2 \|_{L^2(\R^2)} = \|V_g (f_1- f_2) \|_{L^2(\R^2)}= \|f_1-f_2\|_{L^2(\R)} \|g\|_{L^2(\R)}.
$$
Furthermore, an explicit inversion formula for the STFT can be given \cite{grohs2019stable}. %

A question well motivated by applications in audio processing and ptychographic imaging is that of recovering $f$ from partial information of its STFT. More precisely, in these applications, one is interested in recovering a signal $f$ from the magnitude of its STFT, i.e., from $\left|V_g f\right|$. This is a special case of a phase retrieval problem, also referred to as \emph{STFT (or Gabor) phase retrieval.} Note that this is a \emph{nonlinear} recovery problem. While it is not straightforward to resolve the missing phase information, there is, interestingly, a formula that relates $\left|V_g f\right|$ to the \emph{ambiguity function} $\mathcal{A} f$ of $f$, defined as
$$
\mathcal{A} f(x,y) = \int_{\R} f\left(t+x/2\right) \overline{f\left(t-x/2\right)} e^{-2 \pi i t y} \, \rd t,
$$
which can be more compactly written as
$$
\mathcal{A} f(x,y) = e^{\pi i x y} \, V_f f (x,y).
$$
Its relation with the magnitude of the STFT of $f$ is given by (see e.g. \cite{wellershoff2022unknown})
\begin{equation}\label{eqn:amb-func-relation}
\mathcal{F}\left( \left| V_g f \right|^2 \right)(y,-x) = \mathcal{A} f(x,y) \cdot \overline{\mathcal{A} g(x,y)} = V_f f (x,y) \cdot \overline{V_g g(x,y)}\,,
\end{equation}
where $\mathcal{F}$ denotes the two-dimensional Fourier transform on $L^2(\R^2)$, and ``$\cdot$'' denotes pointwise multiplication. We note that the above equation is a direct consequence of the \emph{orthogonality relations} of the STFT (cf.~\cite[Theorem 3.2.1.]{grochenig2001foundations}). %
Thus, at least in theory, the ambiguity function $\mathcal{A} f$ can be recovered from $\left|V_g f\right|$, whenever $\mathcal{A} g$ is non-zero, through
\begin{equation}\label{eqn:pw-inversion}
\mathcal{A} f(x,y)  = \frac{\mathcal{F}\left( \left| V_g f \right|^2 \right)(y,-x)}{\overline{\mathcal{A} g(x,y)}}.
\end{equation}

We note that if a value $c \in \R$ is known, for which $f(c)\neq 0,$ then $f$ can be reconstructed from its ambiguity function via a one-dimensional inverse Fourier transform with respect to the second variable and setting $t=x/2+c:$
\begin{equation}\label{eqn:f-from-Af}
f(x+c) = \frac{1}{\overline{f(c)}} \int_{\R} \mathcal{A} f (x,y) \, e^{2 \pi i (\frac{x}{2} + c) y} \rd y, \quad x \in \R,
\end{equation}
or equivalently,
\[%
f(x+c) = \frac{1}{\overline{f(c)}} \int_{\R} V_f f (x,y) \, e^{2 \pi i (x + c) y} \rd y, \quad x \in \R.
\]

{Existing phase retrieval algorithms} are much more involved than the simple inversion formula~\eqref{eqn:pw-inversion}: the most prominent algorithms employed for a wide range of phase retrieval problems are iterative projection algorithms,  such as the Gerchberg--Saxton algorithm \cite{gerchberg1994practical} and Fienup's algorithm \cite{fienup1982phase},  Averaged Alternating Reflections \cite{bauschke2004finding} and its variants \cite{douglas1956numerical, fannjiang2020fixed, luke2004relaxed, li2017relaxed, chen2018fourier}. A more recent popular iterative scheme is Wirtinger Flow \cite{candes2015phase}. There also exist methods based on convex relaxation, such as PhaseLift \cite{candes2013phaselift}, which is based on the idea of \emph{lifting} the problem to higher dimensions and then recovering a low-rank solution of a linear problem instead of solving the non-linear phase retrieval problem. In practice, however, these algorithms have not been adopted as the iterative methods mentioned before are computationally much more efficient. Another method, specific to STFT phase retrieval, is phase gradient heap integration (PGHI), which recovers the phase of $V_g f$ point by point along a grid on the time-frequency plane \cite{pruvsa2016real}. By far, we have not described the entire literature on phase retrieval algorithms here. For a thorough description, we refer to \cite{fannjiang2020numerics}.

{In contrast,  the simple formula~\eqref{eqn:pw-inversion} is not employed in practice because it is generally very unstable:} typically, one works with time-frequency localized windows because then, the STFT becomes a time-frequency localized signal representation, i.e., it can better describe which frequencies occurred (roughly) at what time (with the limit being the uncertainty principle). At the same time, fast decay of the window $g$ in time or frequency results in fast decay of $\mathcal A g$ in the $x$- or $y$-direction, respectively. Hence, even if $\mathcal{A} g$ is nonzero everywhere on $\R^2$, its magnitude typically decays very fast, making recovery of $\mathcal{A} f$ via Equation~\eqref{eqn:pw-inversion} highly unstable. Take for example the Gauss window $\varphi$. Its ambiguity function is a two-dimensional Gauss function, i.e.,
$$
\mathcal{A} \varphi(x,y) = \frac{1}{\sqrt{2}} e^{-\frac{\pi}{2} (x^2 + y^2)},
$$ 
which decays exponentially as $x,y\rightarrow \pm\infty$. Recovery of $\mathcal{A} f$ via pointwise inversion by $\mathcal{A} \varphi$, as Equation~\eqref{eqn:pw-inversion} suggests, is therefore numerically infeasible, unless the signal $f$ itself also happens to be very well-localized in time and in frequency. 

It is important to emphasize that not only the specific reconstruction formula (\ref{eqn:pw-inversion}) exhibits poor stability properties when used with a single window function, but Gabor phase retrieval itself (from a single window function) is inherently unstable~\cite{alaifari2021gabor}. The stability properties can be linked to the \emph{connectedness} of the Gabor transform magnitude as a function in the time-frequency plane~\cite{grohs2019stable}. Stability can be restored when global recovery of the phase is relaxed and one only aims at recovering the phase on each connected component individually~\cite{alaifari2019stable}. In this work, {by drawing measurements from \emph{several different windows,} we overcome the trade-off between stability and disconnectedness of the measurements.} We aim at \emph{stably} recovering the signal \emph{globally,} instead of on each connected component individually (the latter being inherent to the single window setup). The central question we address in this work is the following:

\smallskip

\noindent\textbf{Question.~}Suppose we can take the STFT magnitude with respect to several windows $\{g_j\}_{j=1}^N$, obtaining a family of measurements $\{\left|V_{g_j} f\right| \}_{j=1}^N$. Can we stably recover $f$ based on the reconstruction formula~\eqref{eqn:pw-inversion} with a collection of properly chosen windows $\{g_j\}_{j=1}^N$?

\smallskip

The rest of the paper is organized to address this key question. In Section~\ref{sec:method}, we propose two multi-window algorithms for STFT phase retrieval --- one utilizing fractional Fourier transformed dilated Gaussian windows, and the other employing Hermite functions of varying degrees. Section~\ref{sec:stability} establishes mathematical results on the improved inverse problem stability provided by these two multi-window approaches. In Section~\ref{sec:numerics}, we present several numerical examples, and conclusions are drawn in Section~\ref{sec:conclusions}.

\section{Multi-window approaches for STFT phase retrieval}\label{sec:method}

In what follows, we propose a framework to use Equation (\ref{eqn:pw-inversion}) constructively, via employing multiple windows $\{g_j\}_{j=1}^N$ with \emph{different time-frequency concentration} each. This way, each set of phaseless measurements $\left| V_{g_j} f \right|$ for a given $j$ allows for stable extraction of a \emph{portion} of $\mathcal{A} f,$ with the region of extraction dictated by the decay of $\mathcal{A} g_j.$ Hence, by acquiring more measurements, we can patch together the different pieces of $\mathcal{A} f$ and globally recover $f$. Note that the ambiguity function has its maximum at the origin $(0,0)$ \cite{grochenig2001foundations}. Therefore, we cannot create windows with ambiguity functions centered at a point different from the origin.

We propose recipes for stable STFT phase retrieval from multiple windows through the following approach:

For some fixed $\varepsilon>0$ and a family of windows $\{g_j\}_{j=1}^N,$ such that $\{\left|V_{g_j} f\right|\}_{j=1}^N$ is measured, proceed as follows:
\begin{itemize} 
    \item For each window $g_j,$ determine the region $\Omega_j \subset \mathbb{R}^2$ such that $\left|\mathcal{A} g_j \right| > \varepsilon$ on $\Omega_j$. Choose the $\Omega_j$'s to be pairwise disjoint.
    \item For each $j,$ recover $\mathcal{A} f$ on $\Omega_j$ by employing the measurements $\left|V_{g_j} f\right|$ in (\ref{eqn:amb-func-relation}).
    \item Combine these partial reconstructions to a global approximation of $\mathcal{A} f$.
    \item Reconstruct an approximation of $f$ (up to a constant global phase factor) from the approximation of its ambiguity function $\mathcal{A} f$ through Equation (\ref{eqn:f-from-Af}).
\end{itemize}

In particular, we propose two architectures for $\left\{g_j\right\}_{j=1}^N$. The first is based on \emph{fractional Fourier transforms of dilated Gaussians,} while the second is based on \emph{Hermite functions of different order.} These choices are based on properties of the ambiguity functions of these windows, {but of course, other window families might also be suitable. In the choices we are proposing, the ambiguity functions of the windows are known explicitly and can be related to special functions. This facilitates the analysis as the ambiguity functions of the windows are guaranteed to be bounded away from zero in certain regions. Alternatively, one could numerically estimate the ambiguity functions to derive different sets of window families that are well suited for multi-window STFT phase retrieval.} 

Before we describe the two measurement systems, we first record some preliminary facts that will be employed.

\subsection{Dilated Gauss windows under the fractional Fourier transform}

To begin with, we give the definition of the \emph{fractional Fourier transform} (FrFT) $\mathcal{F}_\alpha$ of $f \in L^1(\R)$ by the angle $\alpha$, which will be essential in what follows:
\[%
\F_\alpha f (y):= c_\alpha e^{\pi i y^2 \cot \alpha} \int_\R f(t) e^{\pi i t^2 \cot \alpha} e^{-2 \pi i t y / \sin \alpha} \rd t, 
\]
when $\alpha = \R \backslash \pi \mathbb{Z}$. Here, $c_\alpha$ is the square root of $1-i \cot \alpha$ with positive real part. For angles $\alpha = k \pi, k \in \mathbb{Z}$, the fractional Fourier transform is defined as $\F_{k \pi} f:= f$ for $k$ even, and as $\F_{k \pi} f (y) := f(-y) $, for $k$ odd. Note that $\F_{\pi/2} = \F,$ i.e., it is the standard Fourier transform. By a classical density argument, the fractional Fourier transform can be extended to all functions in $L^2(\R)$ \cite{jaming2014uniqueness}.

A property crucial to our first approach is the effect of the fractional Fourier transform on the STFT, which can be described as follows \cite{almeida1994fractional}:
\begin{equation}\label{eqn:STFT-frac-FT}
V_{\F_\alpha g} \F_\alpha f (x,y) = V_g f (R_\alpha(x,y)) e^{\pi i \sin \alpha \left( (x^2-y^2) \cos \alpha -2 x y \sin \alpha \right)}, \quad x,y \in \R,
\end{equation}
for $\alpha \in \R$ and $f,g \in L^2(\R),$ where $R_\alpha(x,y) = (x \cos \alpha - y \sin \alpha, x \sin \alpha + y \cos \alpha)$, i.e., it is the rotation of the vector $(x,y) \in \R^2$ by the angle $\alpha.$ Thus,
$$
\left| V_{\F_\alpha g} \F_\alpha f (x,y)\right| = \left| V_g f (R_\alpha(x,y)) \right|, \quad \mbox{ for } x,y \in \R,
$$
so that $\left| V_{\F_\alpha g} \F_\alpha f \right|$ is a rotation by $\alpha$ of $ \left| V_g f\right|$.

Let us now consider Gauss functions parameterized by a dilation factor $a>0,$
\begin{equation*} %
\varphi^a(t):=a^{1/4} e^{-a \pi t^2},
\end{equation*}
for which $\left\| \varphi^a\right\|_{L^2(\R)} = 2^{-1/4}.$ Computing the ambiguity function of $\varphi^a$ yields
$$
\mathcal{A} \varphi^a(x,y) = \frac{1}{\sqrt{2}} e^{\pi i x y} e^{-\frac{\pi\left( a^2 x^2 + 2 i a x y + y^2 \right)}{2 a}} = \frac{1}{\sqrt{2}} e^{-\frac{\pi}{2} \left( a x^2 + \frac{1}{a}y^2 \right)}.
$$
Hence, the radial decay for the standard Gauss $\varphi = \varphi^1$ is replaced by an elliptical decay with the $x$-axis stretched by the factor $a$. Next, let $\left\{\alpha_j = \pi j/N\right\}_{j=1}^N$ be a set of angles for which we consider the family of window functions 
$$
\phi_{j} := \F_{\alpha_j} \varphi^a,\quad 1 \leq j \leq N,
$$ 
for some fixed dilation factor $a>1$ and some discretization parameter $N$ such that $\{\alpha_{j}\}_{j=1}^N$ is an equidistant discretization of the interval $[0,2\pi]$. In view of Equation~\eqref{eqn:STFT-frac-FT}, this choice of windows results in ambiguity functions $\mathcal{A} \phi_j$ that have an elliptical exponential decay stretched by a factor $a$ along the axis $(\cos \alpha_j,\sin \alpha_j).$ The superposition of these ``essential supports" is shaped in the form of a daffodil and approaches a disc shape by taking $N$ larger, which is also illustrated in~Fig.~\ref{fig:frac_gauss_show}, where all plots are displayed {on a logarithmic scale}.

\begin{figure}[!htbp]
\centering
\subfloat[$a=1$, single window]{\includegraphics[width=0.33\textwidth]{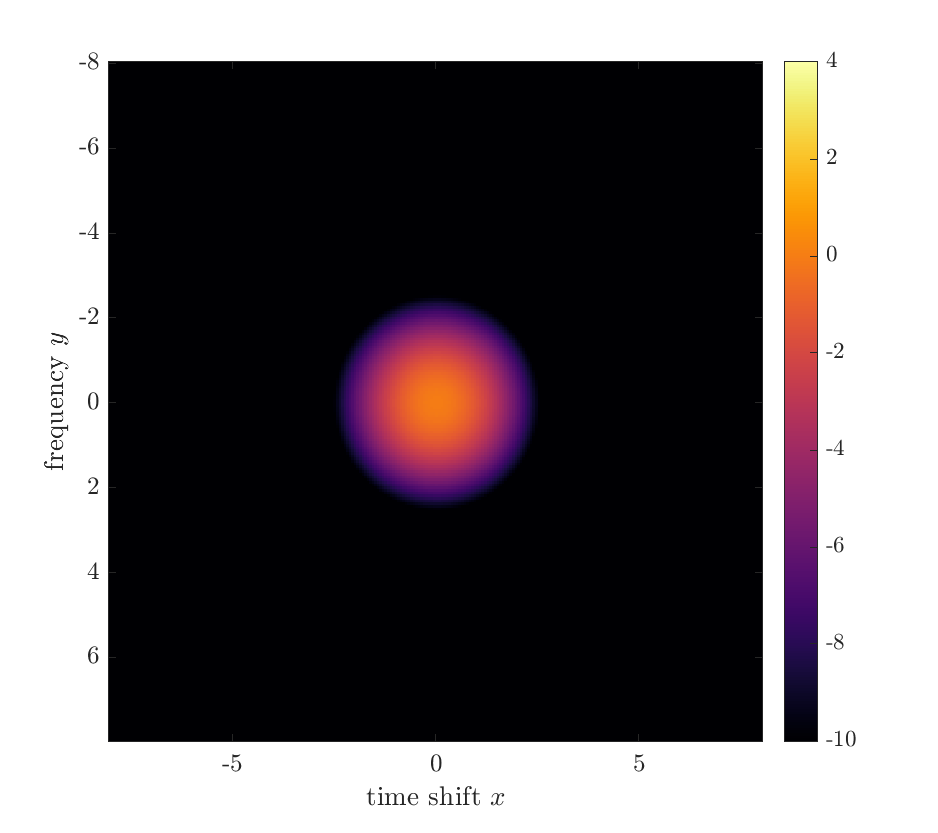}\label{fig:frac_gauss_show_1}}
\subfloat[$a=2$, single window]{\includegraphics[width=0.33\textwidth]{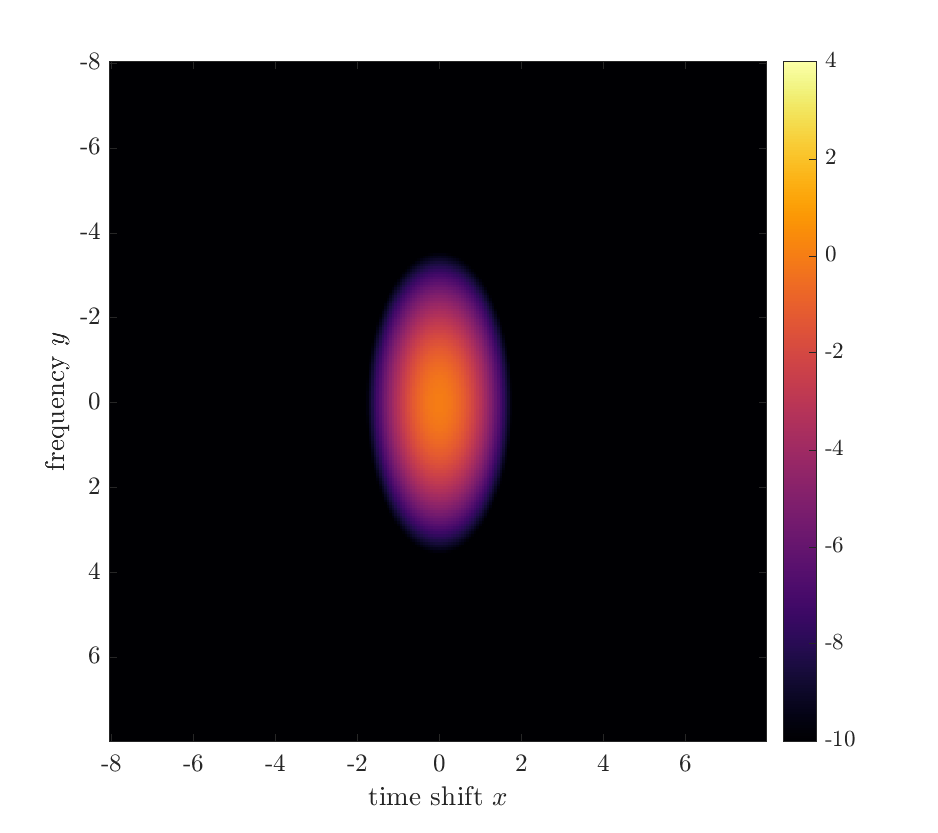}}
\subfloat[$a=2$, $8$ windows]{\includegraphics[width=0.33\textwidth]{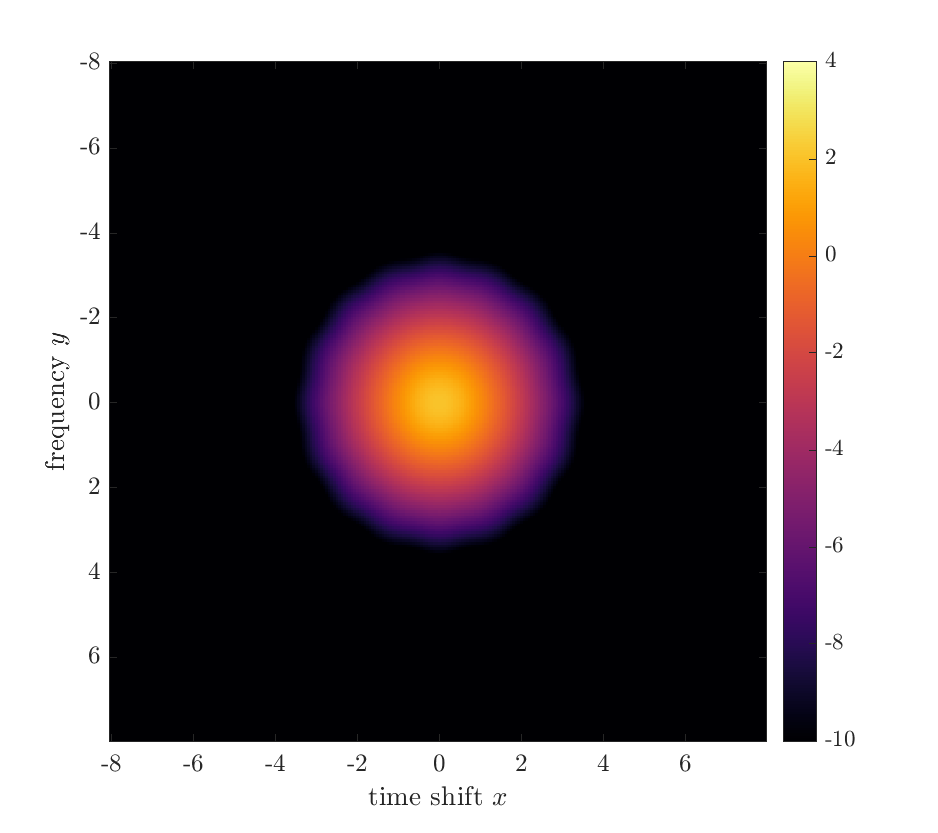}}\\
\subfloat[$a=10$, single window]{\includegraphics[width=0.33\textwidth]{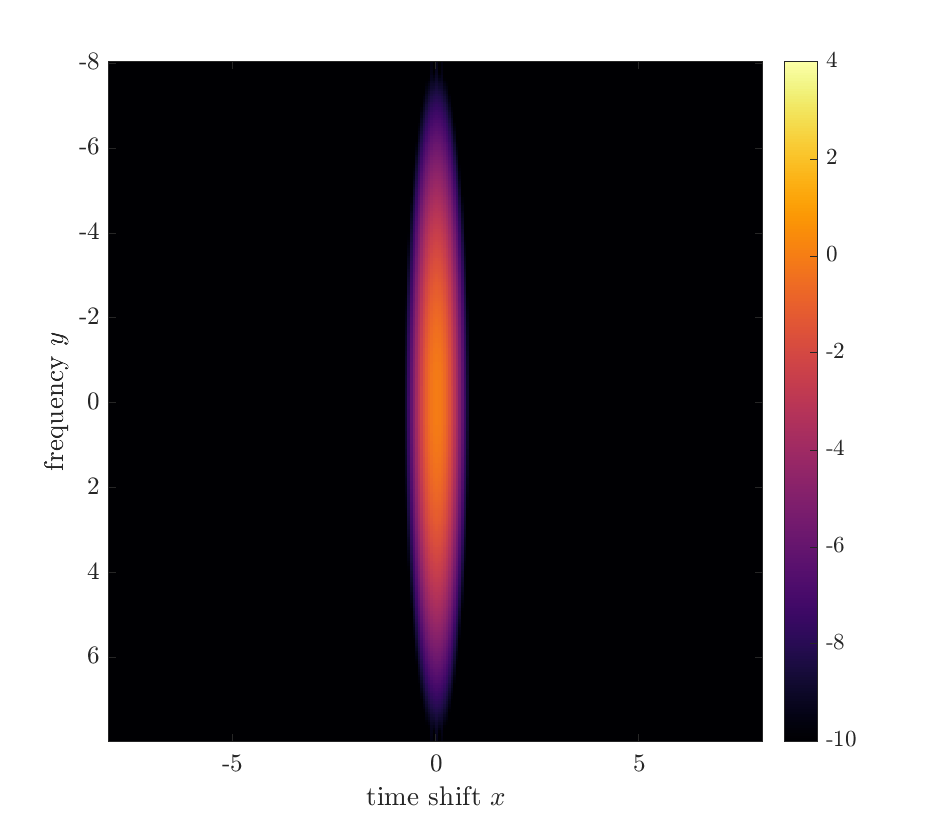}}
\subfloat[$a=10$, $4$ windows]{\includegraphics[width=0.33\textwidth]{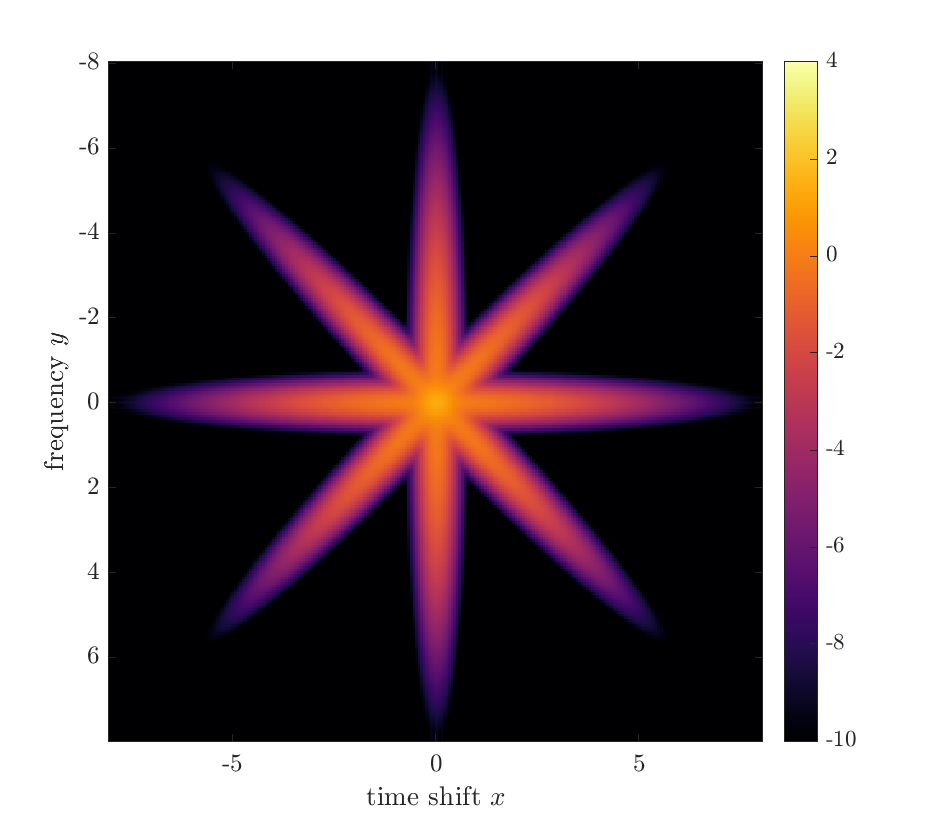}\label{fig:daffodil}}
\subfloat[$a=10$, $40$ windows]{\includegraphics[width=0.33\textwidth]{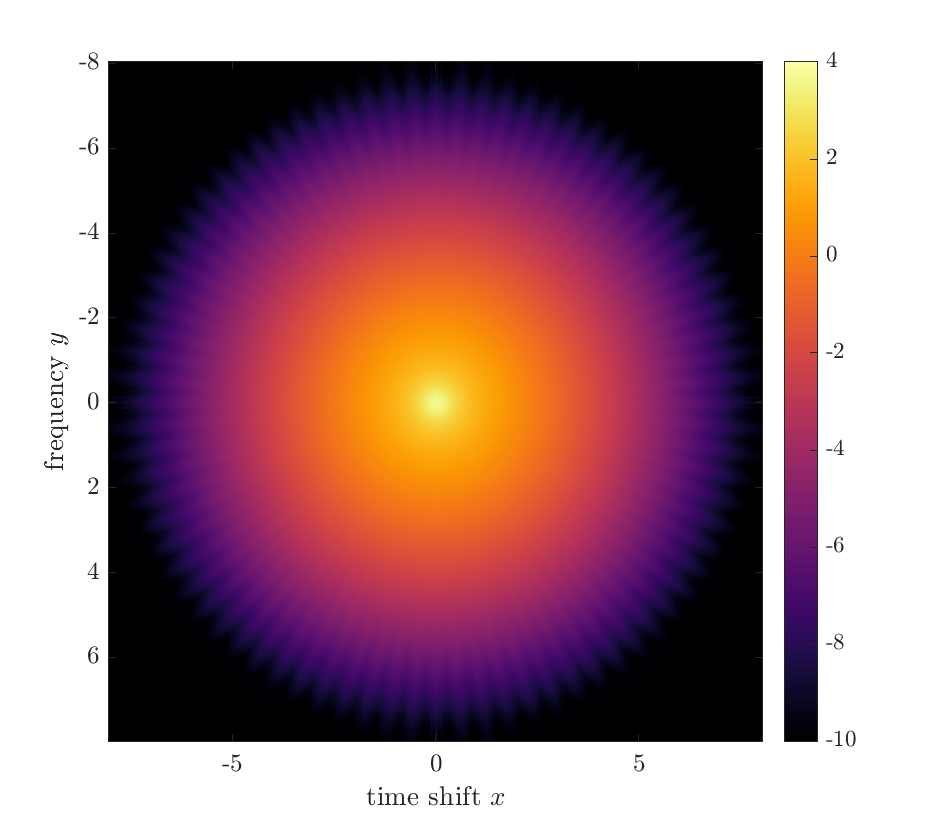}\label{fig:disc}}
\caption{The absolute value of (the sum of) ambiguity functions for fractional Fourier transformed (dilated) Gauss windows. All plots are displayed on log-scale. (a) $\mathcal{A} \varphi^1$, where $\varphi^1$ is the standard Gauss; (b) $\mathcal{A} \varphi^2$ for the dilated Gauss with $a=2$; (c) $\sum_{j=1}^N \mathcal{A}(\mathcal{F}_{\tilde{\alpha}_j} \varphi^2)$ with $N=8$; (d) $\mathcal{A} \varphi^{10}$ for the dilated Gauss with $a=10$; (e) $\sum_{j=1}^N \mathcal{A}(\mathcal{F}_{{\alpha}_j} \varphi^{10})$ with $N=4$; (f) $\sum_{j=1}^N \mathcal{A}(\mathcal{F}_{{\alpha}_j} \varphi^{10})$ with $N=40$.
}\label{fig:frac_gauss_show}
\end{figure}

Note that the ``leaves" of the daffodil grow in length but get narrower in width as $a$ becomes larger. Thus, to approximately cover a disc of a larger diameter, the number of windows  $N$ has to grow too. We also remark that the dilated Gauss function is not the only possible choice. The main property that is exploited here, is the fact that the ambiguity functions of these (dilated) Gauss functions have no zeros and are strictly lower bounded on a predefined region (which we choose to be elliptical). A result in~\cite{grochenig2020zeros} states that the Gauss function is not the only window with zero-free ambiguity function. Another example is the \emph{one-sided exponential.}

\begin{remark}
For given parameters $a, \varepsilon$ and $N$, one can quantify the largest disc that can be covered by the rotated ellipses, therefore guaranteeing that $\mathcal{A} f$ will be recovered on a disc $B_R(0)$ for some $R.$ More precisely, this radius is determined by the points of intersection between the ellipses. For this, it suffices to consider a point that lies on the ellipse parametrized by
\begin{equation}\label{eqn:ellipse}
a x^2 + \frac{y^2}{a} = \frac{2}{\pi} \left| \ln (\sqrt{2} \varepsilon) \right|
\end{equation}
and its rotated version with rotation angle $\pi/N.$ One point of intersection is then given by intersecting the original ellipse with the line parametrized by $(r \cos(\frac{\pi}{2}-\frac{\pi}{2 N}), r \sin(\frac{\pi}{2}-\frac{\pi}{2N}))=(r \sin(\frac{\pi}{2N}), r \cos(\frac{\pi}{2N}))$. Plugging this into (\ref{eqn:ellipse}) results in the radius
$$
R_1^2 = \frac{C}{D},
$$
where 
$
C =\frac{2}{\pi}\Bigl\lvert \ln\bigl(\sqrt{2}\,\varepsilon\bigr)\Bigr\rvert$, $D =\frac{1}{a}\cos^2\!\left(\tfrac{\pi}{2N}\right)+ a \,\sin^2\!\left(\tfrac{\pi}{2N}\right)$. A corresponding point of intersection is given by
$$
(\widetilde{x},\widetilde{y}) = \left( \sqrt{\frac{C}{D}} \sin{\frac{\pi}{2N}}, \sqrt{\frac{C}{D}} \cos \frac{\pi}{2N} \right).
$$
If, instead, one considers the other pair of intersecting points, one obtains the radius
$$
R_2 =\sqrt{\frac{C}{a\,\cos^2\!\Bigl(\tfrac{\pi}{2N}\Bigr)+\frac{\sin^2\!\bigl(\tfrac{\pi}{2N}\bigr)}{a}}}.
$$
Assuming $a>1$ and $N\geq 3$, we have $R_1 > R_2$. The union of the set of ellipses covers the larger disc. By symmetry, the ambiguity function $\mathcal A f$ is then guaranteed to be recovered on $B_{R_1}(0)$. 

\end{remark}

The daffodil-shaped geometry shown in~Fig.~\ref{fig:daffodil} is a very simple construction and the reconstruction can be further simplified as follows. A simple calculation yields that the ambiguity functions of $\phi_j$ are real-valued and positive. Here, we use the notation $R_{\alpha_j}(x,y)=(x(\alpha_j),y(\alpha_j))$:
\begin{align*}
    \mathcal{A} \phi_j(x,y) &= e^{\pi i x y} V_{\phi_j} \phi_j(x,y)\\
    &= e^{\pi i x y} e^{\pi i \sin \alpha_j \left( (x^2-y^2) \cos \alpha_j - 2 x y \sin \alpha_j \right)} V_{\varphi^a} \varphi^a \left( R_{\alpha_j} (x,y)\right)\\
    &= e^{\pi i x(\alpha_j) y(\alpha_j)} V_{\varphi^a} \varphi^a (x(\alpha_j),y(\alpha_j))\\
    &= \mathcal{A} \varphi^a(R_{\alpha_j} (x,y)).
\end{align*}
Since $\mathcal{A} \varphi^a$ is real-valued and positive, this is also true for $\mathcal{A} \phi_j$. Therefore, we can directly sum the reconstruction formula (\ref{eqn:amb-func-relation}) over $j=1,\dots,N:$
\begin{equation}\label{eqn:summed-amb-func-relation}
\sum_{j=1}^N \mathcal{F}\left( \left| V_{\phi_j} f \right|^2 \right)(y,-x) = \mathcal{A} f(x,y) \cdot \sum_{j=1}^N \mathcal{A} \phi_j(x,y)\,,
\end{equation}
and determine $\mathcal{A} f$ on the domain 
$$
\Omega = \left\{ (x,y) \in \mathbb{R}^2: \sum_{j=1}^N \mathcal{A} \phi_j(x,y) > \varepsilon \right\}
$$
through pointwise division in (\ref{eqn:summed-amb-func-relation}). We propose~\Cref{alg1} for the multi-window inversion using fractional Fourier transformed Gaussian windows at various angles. 
\begin{algorithm} [ht!]
  \caption{A Multiscale STFT Phase Retrieval Algorithm Based on the Fractional Fourier Transformed Gaussian Windows.} \label{alg1}
\begin{algorithmic}[1]
\State Given $|V_{\mathcal{F}_{\alpha} \varphi^a} f|$ and $\mathcal{A}(\mathcal{F}_{\alpha} \varphi^a)$ where $\varphi^a$ is the dilated Gaussian, and $\mathcal{F}_{\alpha}$ denotes the Fractional Fourier transform. We consider measurements from a collection of angles $\{\alpha_j\}_{j=1}^N$ and define $\phi_j = \mathcal{F}_{\alpha_j} \varphi^a$. Let $\varepsilon$ be the tolerance.
\For{$j = 1$ to  $N$ }
\State Compute $ \phi_j(x,y) = \mathcal{F}(|V_{\phi_j} f|^2)$ (2D FFT)  and $\widetilde{G}_j(x,y) = G_j(y,-x)$.
\EndFor 
\State Find $\Omega = \{(x,y): \sum_{j=1}^N \mathcal{A} \phi_j > \varepsilon \}$.
\State Define $J(x,y)= \sum_{j=1}^N \widetilde{G}_j(x,y)  / \sum_{j=1}^N \mathcal{A}\phi_j(x,y) $, for $(x,y)\in \Omega$ and  $J(x,y) = 0$ on $\Omega^c$.
\State Compute $\widehat{J}(x,\cdot) = \mathcal{F}^{-1} ( J(x,\cdot) )$ (1D FFT).
\State Return the reconstructed signal $\widetilde{f}(x+c) =  \widehat{J}(x,x/2+c) /{\overline{f(c)}}$ (assuming $f(c)\neq 0$).%
\end{algorithmic}
\end{algorithm}

A drawback of this setup is that many windows are required to achieve a good support coverage of the time-frequency plane by their ambiguity functions, such as the disc shown in~Fig.~\ref{fig:disc}. Next, we propose a more technical construction, but with the advantage of requiring less data in the sense that fewer windows are necessary for a good support coverage.

\begin{figure}[!htbp]
\centering
\subfloat[Hermite function $h_1$]{\includegraphics[width=0.33\textwidth]{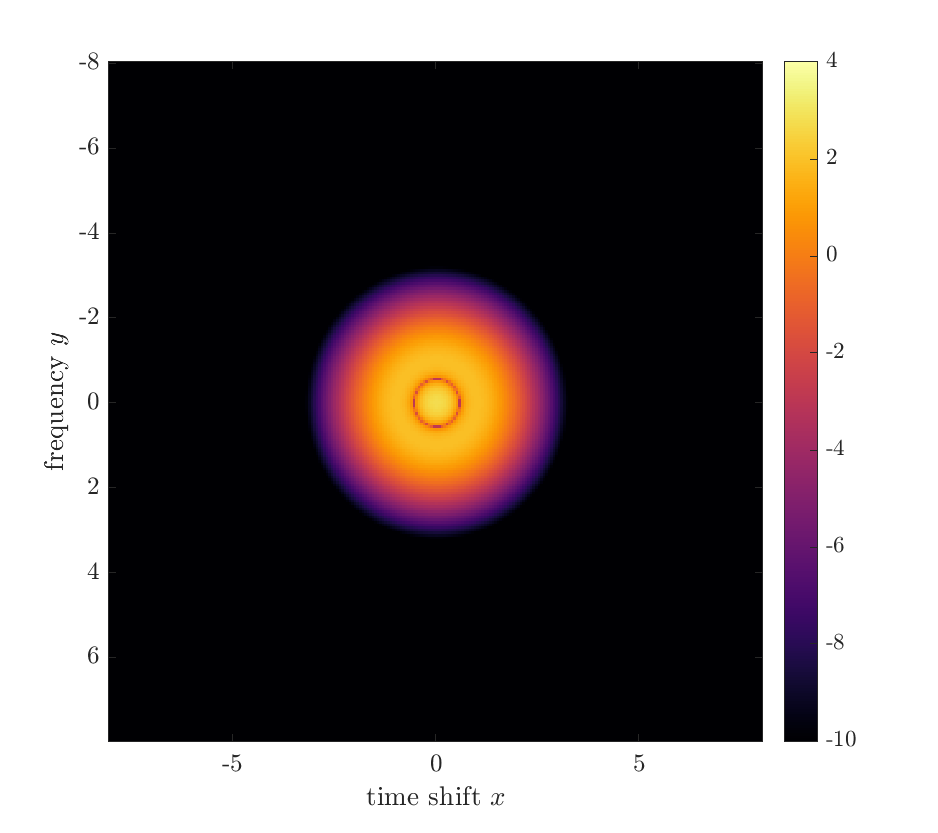}\label{fig:2a}}
\subfloat[Hermite function $h_5$]{\includegraphics[width=0.33\textwidth]{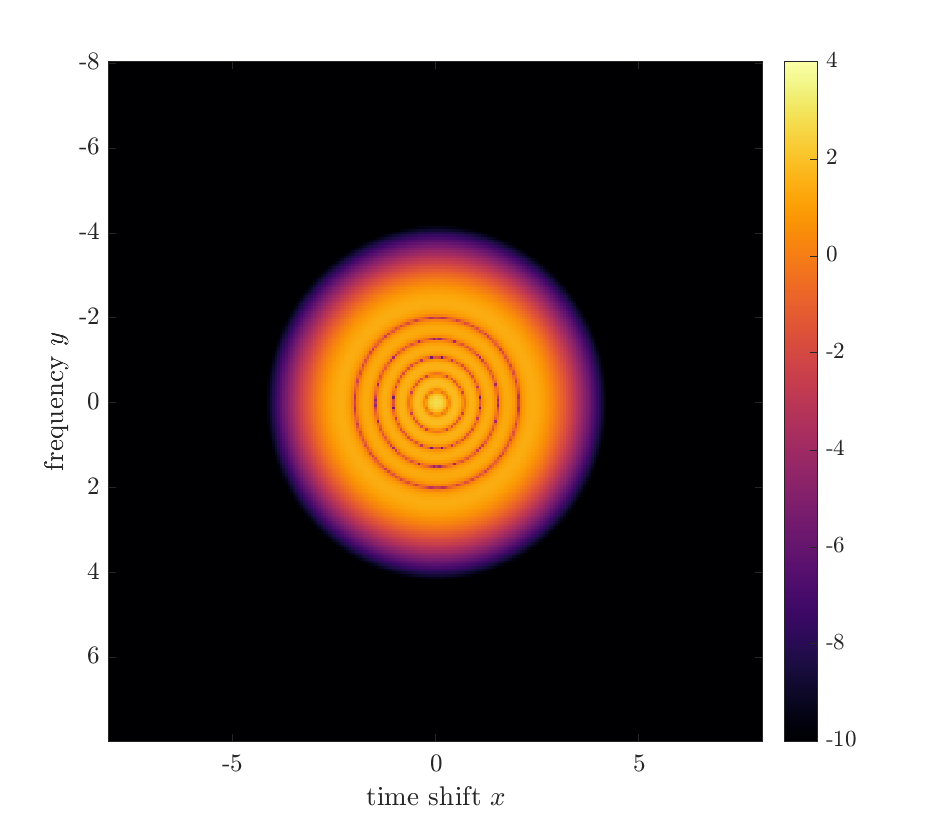}\label{fig:2b}}
\subfloat[Hermite function $h_{10}$]{\includegraphics[width=0.33\textwidth]{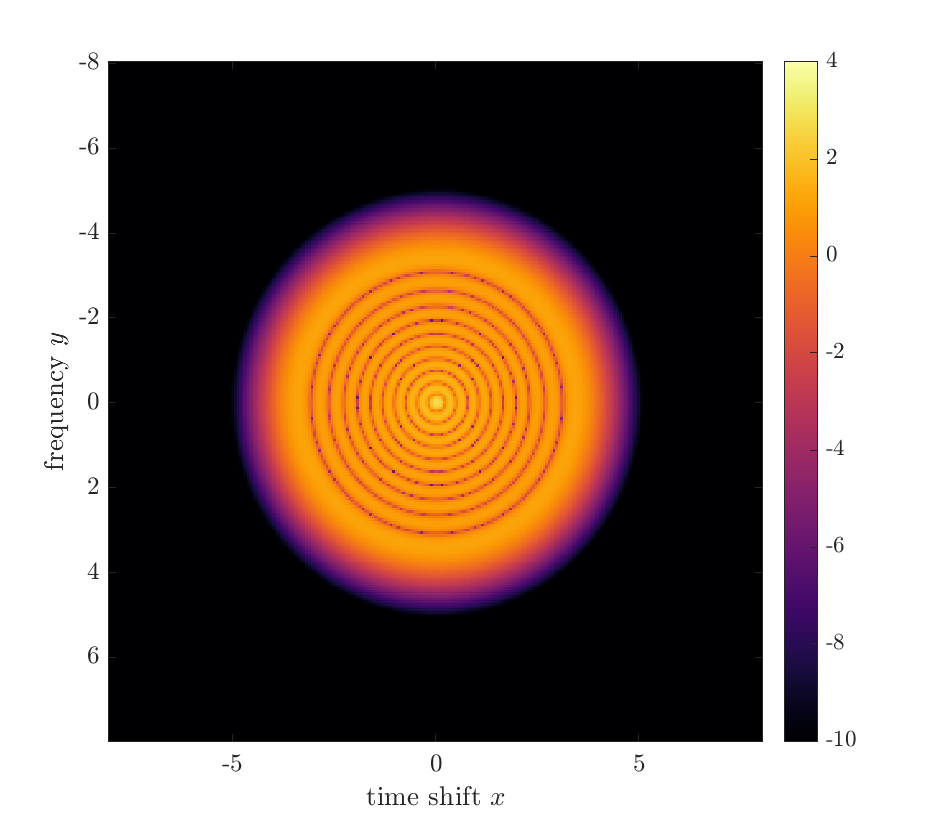}\label{fig:2c}}\\
\subfloat[$S_n^\varepsilon$ with $n=1$]{\includegraphics[width=0.33\textwidth]{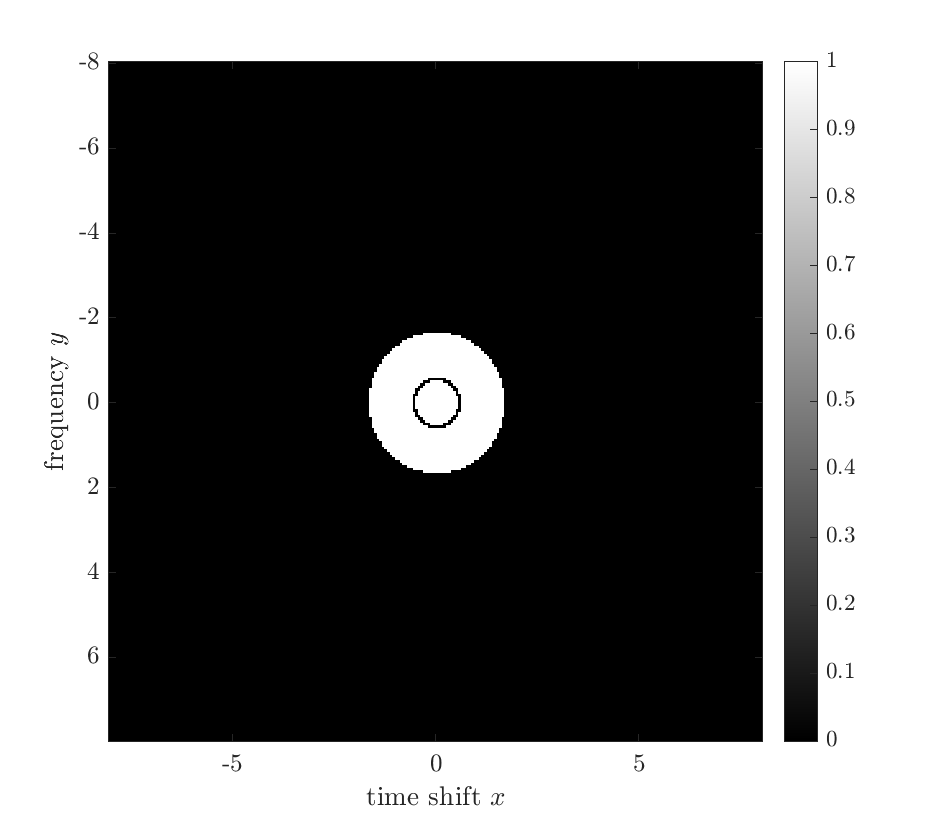}\label{fig:2d}}
\subfloat[$S_n^\varepsilon$ with $n=5$]{\includegraphics[width=0.33\textwidth]{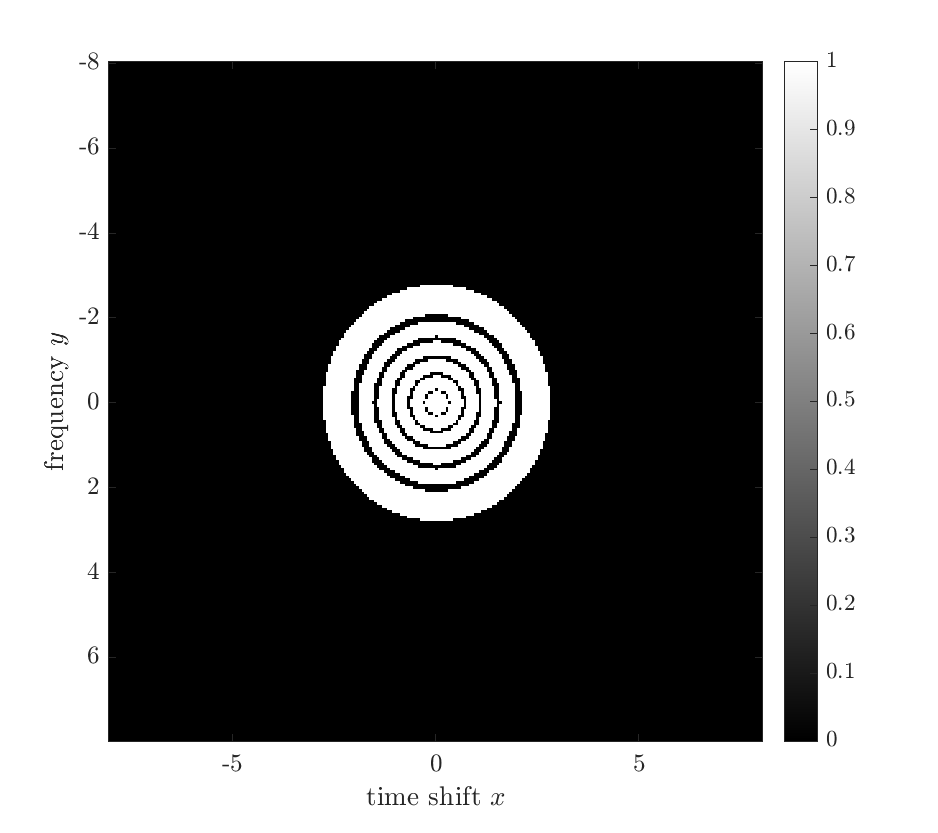}\label{fig:2e}}
\subfloat[$S_n^\varepsilon$ with $n=10$]{\includegraphics[width=0.33\textwidth]{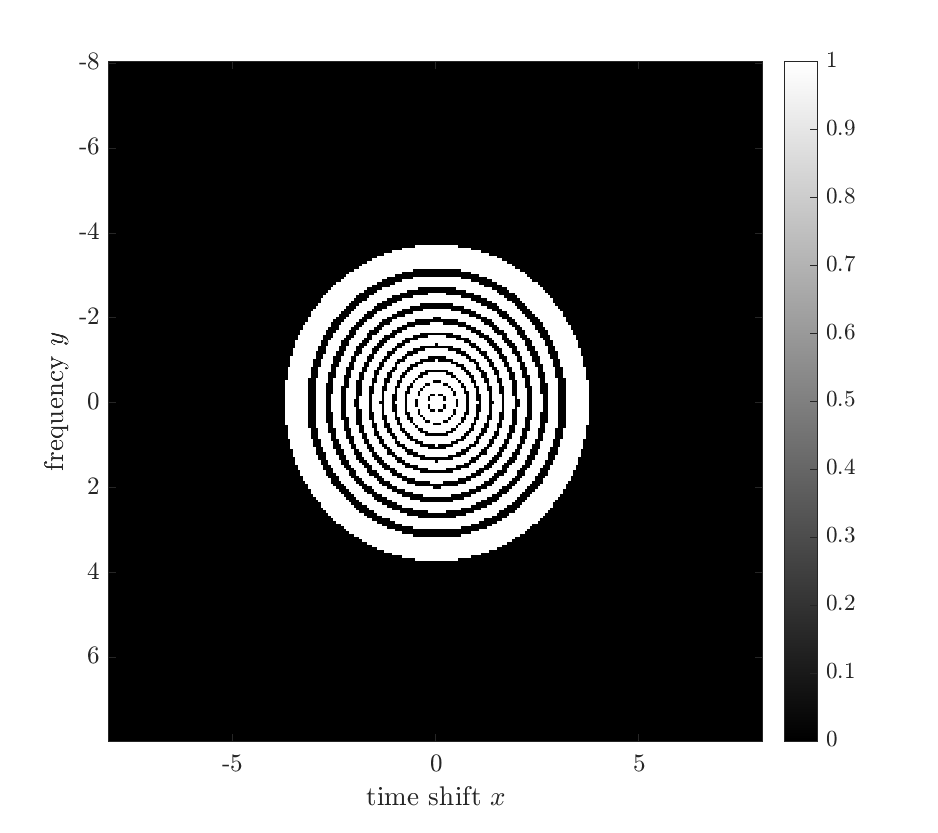}\label{fig:2f}}\\
\subfloat[$\bigcup S_n^\varepsilon$ with $n=1,5$]{\includegraphics[width=0.33\textwidth]{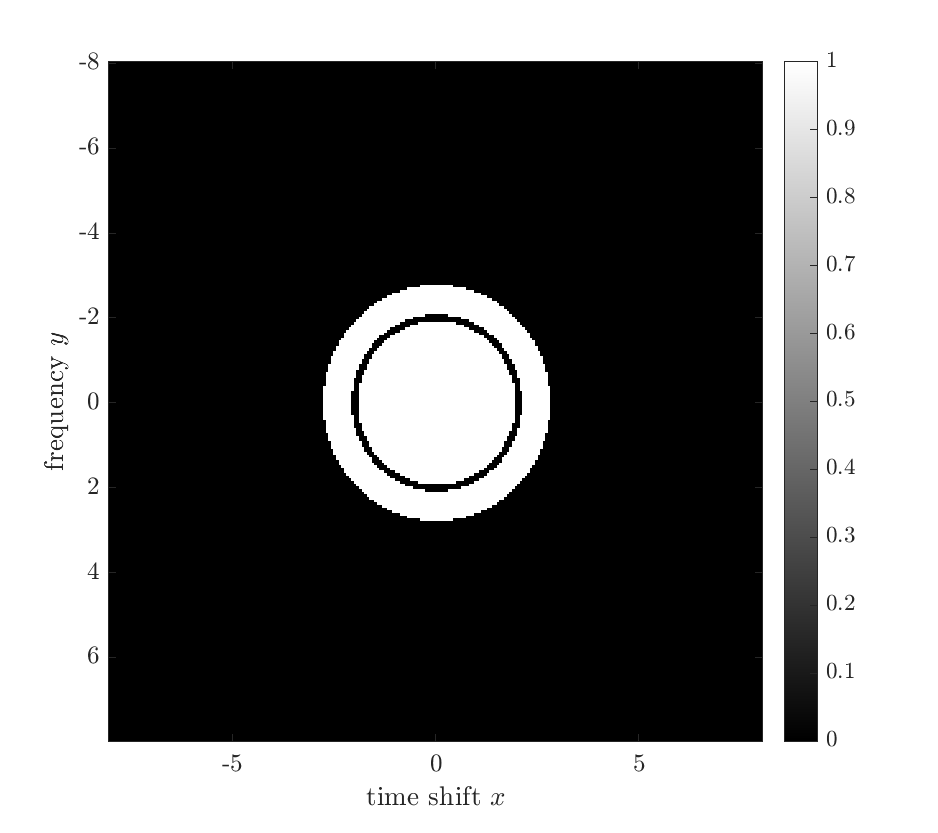}\label{fig:2g}}
\subfloat[$\bigcup S_n^\varepsilon$, $n=1,5,10$]{\includegraphics[width=0.33\textwidth]{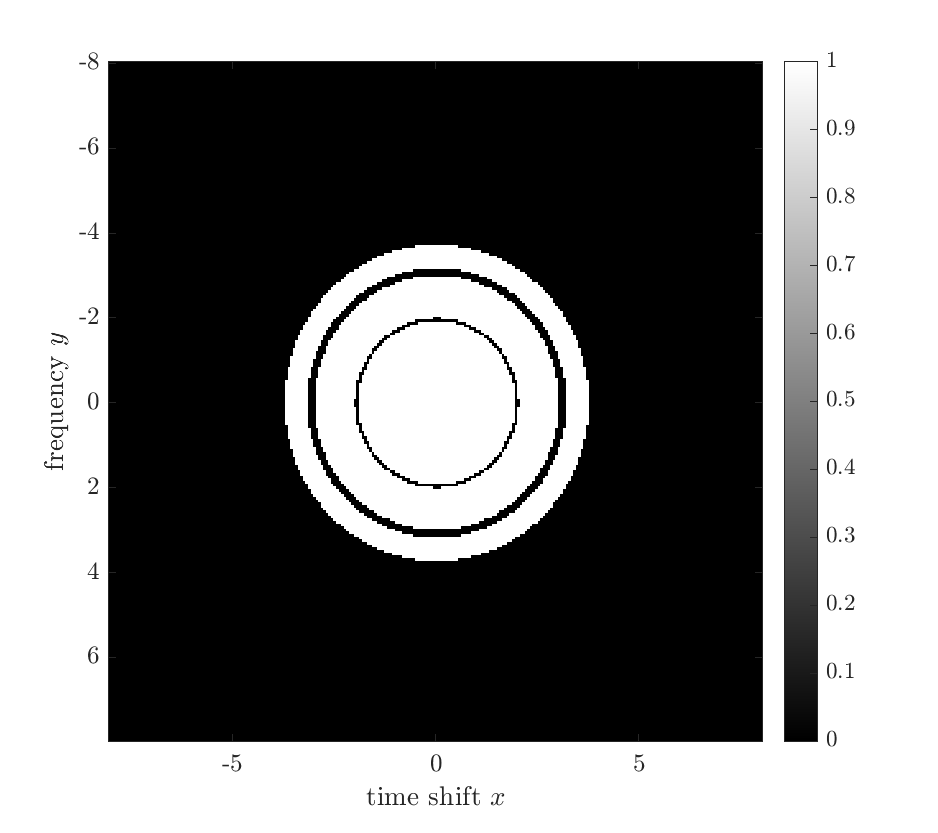}\label{fig:2h}}
\subfloat[$\bigcup S_n^\varepsilon$ $n =1$,$3$,$5$,$7$,$10$]{\includegraphics[width=0.33\textwidth]{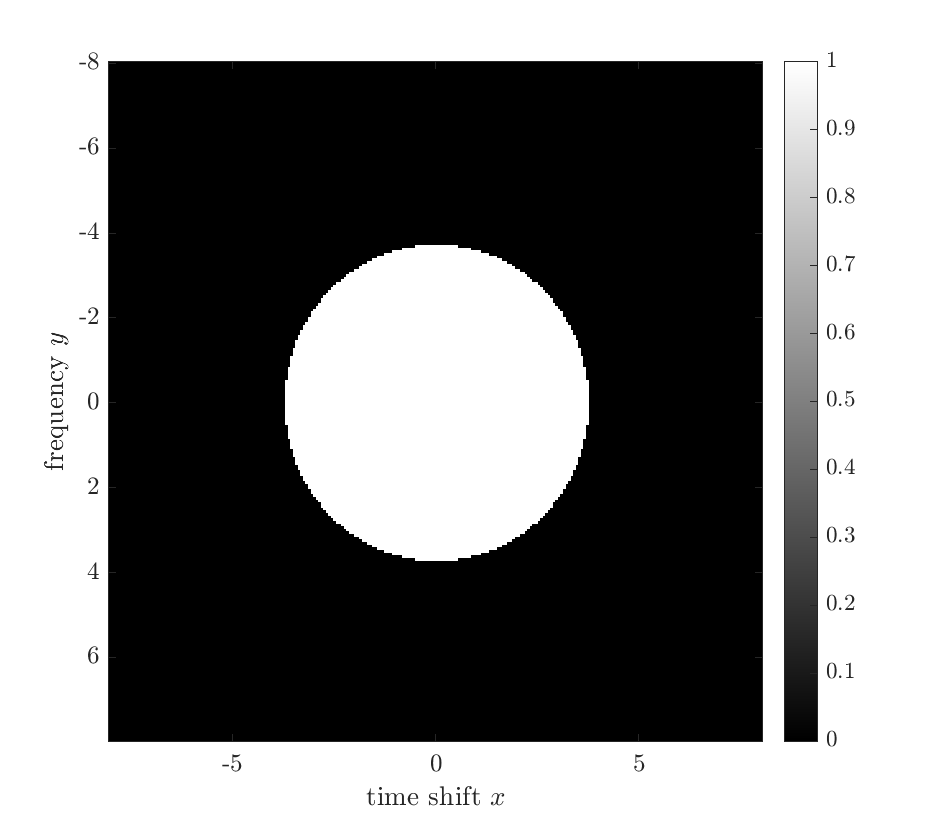}\label{fig:2i}}
\caption{Top row (a)-(c): the absolute value of the Hermite function ambiguity function plotted in the log scale for $n=1,5,10$ , respectively. Middle row (d)-(f): the white region reflects the stability set $S_n^\varepsilon$ defined in~\eqref{eq:stability_set} with $\varepsilon = 0.1$ for Hermite functions of degree $n=1,5,10$, respectively. Bottom row (g)-(i): the white region reflects the union of the stability set $\bigcup S_n^\varepsilon$ in~\eqref{eq:stability_set} with $\varepsilon = 0.1$ for Hermite functions with different sets of the indices $n$.\label{fig:hermite_show}}
\end{figure}

\subsection{Hermite windows}

For this second architecture, the set of windows will consist of Hermite (also called Hermite--Gaussian) functions of increasing order. Their ambiguity functions have rotational symmetry with a number of circular root sets related to the order of the Hermite function; see~Fig.~\ref{fig:2a}, Fig.~\ref{fig:2b} and Fig.~\ref{fig:2c}. The idea of our approach is to recover $\mathcal{A} f$ as if one was ``peeling an onion'': for each Hermite function, we define an annulus on which $\mathcal{A} f$ can be retrieved stably. The union of these annuli results in a large disc in the time-frequency plane. In what follows, we will make this procedure more precise. For this, we start with the definition of the Hermite function $h_n$ of order $n$~\cite{abreu2014function}:

\begin{definition}
For $n \in \mathbb{N},$ the $n$-th Hermite function is defined as
\begin{equation*}
    h_n(t):= c_n e^{\pi t^2 } \frac{d^n}{d t^n} e^{- 2 \pi t^2},
\end{equation*}
where $c_n$ is a constant such that $\|h_n\|_{L^2(\mathbb{R})} = 1$. In particular, $h_0(t) = 2^{1/4} e^{-\pi t^2}$, which is the $L^2-$normalized standard Gauss function. 
\end{definition}

We summarize a few basic properties of the Hermite functions in the following.
\begin{lemma}[cf. \cite{namias1980fractional}]
The Hermite functions $h_n$ are $L^2$-normalized, i.e., $\|h_n\|_{L^2(\mathbb{R})} = 1$ and $h_n$ is the $n$-th eigenfunction of the Fractional Fourier transform with
\begin{equation*}%
\mathcal{F}_\alpha h_n = e^{-i \alpha n} h_n.
\end{equation*}
The family $\left\{ h_n \right\}_{n\in \mathbb{N}}$ forms an orthonormal basis of $L^2(\mathbb{R})$. In particular, the Fourier transform of $h_n$ is $(-i)^n h_n.$
\end{lemma}

The ambiguity function of the $n$-th Hermite function is~\cite{alaifari2021uniqueness}
\begin{equation}\label{eq:Hermite_amb}
\mathcal{A} h_n(x,y) = e^{-\pi (x^2+ y^2) /2} L_n \left( \pi (x^2 + y^2)  \right), 
\end{equation}
where $L_n (z)$, $z\in\mathbb{R},$ is the Laguerre polynomial 
\[%
L_n(z) := \sum_{k=0}^n \frac{n!}{(n-k)! k!} \frac{(-z)^k}{k!}.
\]
We further define the $n$-th Laguerre function $\mathfrak{L}_n(z)$ as
\[%
    \mathfrak{L}_n(z) = e^{ - z/2}   L_n(z),\,\quad z\in \mathbb{R}.
\]
The set of Laguerre functions $\{\mathfrak{L}_n(z)\}_{n \in \mathbb{N}}$ forms an orthogonal basis of $L^2(\mathbb{R})$ and {can be used to express the ambiguity function in (\ref{eq:Hermite_amb}) as}
\[
\mathcal{A} h_n(x,y) = \mathfrak{L}_n( \pi (x^2 + y^2)).
\]
The $n-$th Laguerre polynomial, and hence also the $n-$th Laguerre function, has exactly $n$ (strictly positive, pairwise different) roots. By Equation (\ref{eq:Hermite_amb}), $\mathcal{A} h_n$ is rotation-invariant. Hence, the root set of $\mathcal{A} h_n$ consists of $n$ disjoint circles, centered at the origin $(0,0)$ and thus has  Lebesgue measure zero. {For stable phase retrieval, we need to determine regions on which $\mathcal{A} h_n(x,y)$ is lower bounded in absolute value.}

To this end, we consider the stability set $S_n^\varepsilon$ for any $\varepsilon \in (0,1)$ defined by
\begin{equation}\label{eq:stability_set}
S_n^\varepsilon := \{ (x,y)\in \mathbb{R}^2:  |\mathcal{A} h_n(x,y)| > \varepsilon\} = \{ (x,y)\in \mathbb{R}^2:  |\mathfrak{L}_n ( \pi (x^2 + y^2))| > \varepsilon\}.
\end{equation}
In~Fig.~\ref{fig:2d}, Fig.~\ref{fig:2e} and Fig.~\ref{fig:2f}, the white region reflects the set $S_n^\varepsilon$ with $\varepsilon = 0.1$ and $n=1,5,10$, respectively. 
The stability set $S_n^\varepsilon$ is determined by the subset of $\mathbb{R}$ on which the $n$-th Laguerre function $\mathfrak{L}_n$ has a large absolute value. Hence, better estimates of $S_n^\varepsilon$ can aid the design of an efficient numerical strategy for phase retrieval by choosing a minimal set of Hermite windows $\{h_{n_k}\}$, such that $\bigcup_k S_{n_k}^\varepsilon$ is as large as possible. As shown in~Fig.~\ref{fig:2g}, Fig.~\ref{fig:2h} and Fig.~\ref{fig:2i}, we plot the union of the stability set $\bigcup_{n\in I} S_{n}^\varepsilon$ for different sets of indices $I$. The larger the set $I$, the larger the union of the stability sets on which we can stably recover $\mathcal{A} f(x,y)$.

We propose~\Cref{alg2} for the multi-window STFT phase retrieval using Hermite functions of various degrees as the measuring windows.

\begin{algorithm}[ht!]
  \caption{A Multiscale Algorithm for the Hermite Functions} \label{alg2}
\begin{algorithmic}[1]
\State Given the phaseless STFT measurements $|V_{h_{n_j}} f|$ and the ambiguity functions $\mathcal{A} h_{n_j}$ where $h_{n_j}$ is the $n_j$--th Hermite function, $j = 1,\ldots, N$, with $n_1 < \cdots < n_N$. Let $\Omega = \emptyset$ and the tolerance be $\varepsilon>0$.
\For{$j = 1$ to  $N$ }
\State Find $\Omega_j = \{(x,y): |\mathcal{A} h_{n_j}| > \varepsilon \}  \setminus \Omega$ and  update $\Omega \leftarrow \Omega \cup \Omega_j $.
\State Compute $ G_j(x,y) = \mathcal{F}(|V_{h_{n_j}} f|^2)$ (using 2D FFT) and set $\widetilde{G}_j(x,y) = G_j(y,-x)$.
\State Define $J(x,y)= \widetilde{G}_j(x,y)  / \overline{\mathcal{A} h_{n_j}(x,y)}$, for $(x,y)\in \Omega_j$.
\EndFor 
\State Set $J=0$ on $\Omega^c$. Compute $\widehat{J}(x,\cdot) = \mathcal{F}^{-1} ( J(x,\cdot) )$ (using 1D inverse FFT).
\State Return the recovered function $\widetilde{f}(x+c) =   \widehat{J}(x,x/2+c) /{\overline{f(c)}}$ (assuming $f(c)\neq 0$).
\end{algorithmic}
\end{algorithm}

\begin{remark}
    If we are given $V_g g$ rather than $\mathcal{A} g$ where $g$ is the window function, we can still use~\Cref{alg1,alg2} with $\mathcal{A} g$ therein replaced by $V_g g$ and the Step 8 of both algorithms changed to $\widetilde{f}(x+c) = \widehat{J}(x,x+c) /{\overline{f(c)}}$, assuming $f(c)\neq 0$.
\end{remark}

\section{Stability estimates}\label{sec:stability}
The improved stability properties of our approach, evidenced by numerical examples in Section~\ref{sec:numerics}, can be quantified more explicitly. We first derive an estimate for~\Cref{alg1}. A simple modification, outlined in Section~\ref{sec:stab-alg2}, implies a similar result for~\Cref{alg2}.

For this, suppose that on each region $\Omega_j \subset \R^2$ we compare a set of \emph{exact} measurements
$$
\left\{ F_j := \left| V_{g_j} f \right|^2 \right\}_{j=1}^N \,,
$$
and measurements corrupted by additive noise
$$
\left\{ F_j^\eta:= \left| V_{g_j} f \right|^2 + \eta_j \right\}_{j=1}^N\,,
$$
where $\{g_j\}_{j=1}^N$ is a set of $N$ distinct window functions. The goal is to analyze the stability of reconstructing $f$ from $\left\{F_j\right\}_{j=1}^N$ via formulas (\ref{eqn:pw-inversion}) and (\ref{eqn:f-from-Af}), where we employ (\ref{eqn:pw-inversion}) with $F_j, F_j^\eta$ on  $\Omega_j$ for each $j=1, \dots, N$. 

There are two error terms that enter in the analysis: the reconstruction error due to the fact that the ambiguity function will only be determined on the bounded domain $\Omega : = \bigcup_j \Omega_j$ and the propagation of the noise terms $\eta_j$. We will denote by $\f$ and $\feta$ the reconstructions from $F_j$'s and $F^\eta_j$'s, respectively. The two error terms we need to bound are on the right-hand side (RHS) of
\begin{equation*}%
\left\|f-\feta\right\|_{L^p(\mathbb{R})} \leq \left\|f-\f\right\|_{L^p(\mathbb{R})} + \left\|\f-\feta\right\|_{L^p(\mathbb{R})},
\end{equation*}
where the first term on the RHS is the approximation error and the second is the propagation of data noise. 

\subsection{Stability of Algorithm~\ref{alg1}}\label{sec:stab-alg1}
This subsection analyzes the stability of Algorithm~\ref{alg1}.

\subsubsection{Propagation of data noise}
We start by estimating the second term, $\|\f-\feta\|_{L^p(\R)}$. For this, let $A$ and $A^\eta$ denote the functions obtained on $\Omega = \bigcup_{j=1}^N \Omega_j$ through the reconstruction formula (\ref{eqn:summed-amb-func-relation}) from $F_j,$ $F_j^\eta$, respectively. That is,
\[
  A(x,y) = \frac{\sum_{j=1}^N \mathcal{F}\left( F_j  \right)(y,-x)}{\sum_{j=1}^N \mathcal{A} \phi_j(x,y)}, \quad
  A^\eta(x,y) = \frac{\sum_{j=1}^N \mathcal{F}\left(  F_j^\eta \right)(y,-x)}{\sum_{j=1}^N \mathcal{A} \phi_j(x,y)}\,, \quad (x,y) \in \Omega.
\]

Since the reconstructions are only computed on $\Omega,$ we simply set $A$ and $A^\eta$ to zero on $\R^2\backslash \Omega.$  This does not affect the data noise propagation, but the fact that the ambiguity function will only be recovered on $\Omega$ will enter in the reconstruction error term later. To this end, we introduce the mixed-norm $L^{p,1}$ space consisting of all functions $U$ on $\R^2$ for which
$$
\left\|U \right\|_{L^{p,1}(\R^2)} := \int_\R \left(\int_\R  \left| U(x,y) \right|^p \rd x \right)^{1/p}  \rd y  < \infty\,.
$$
We derive the following bounds on the propagation of data noise in the ambiguity functions.
\begin{lemma}
Let $A, A^\eta$ be the reconstructions obtained by (\ref{eqn:summed-amb-func-relation}) from $F_j,$ $F_j^\eta,$ respectively, $j=1,\dots, N$ and for some $\varepsilon > 0$, let $\Omega$ be chosen such that $\sum_{j=1}^N \mathcal{A}\phi_j (x,y) > \varepsilon$ for all $(x,y) \in \Omega$. Then,
\begin{align}\label{eq:A-eta-gauss}
\left\|A-A^\eta\right\|_{L^{p,1}(\R^2)} & \leq \frac{C_p}{\varepsilon}   \left\|\sum_{j=1}^N \eta_j\right\|_{{L^1}(\mathbb{R}^2)},
\end{align}
where {$C_p:= \left\|\mathds{1}_\Omega \right\|_{L^{p,1}(\R^2)} < \infty$ and} $\mathds{1}_\Omega$ denotes the characteristic function on $\Omega$. Furthermore, 
\begin{align}
\left\|A(0,\cdot)-A^\eta(0,\cdot)\right\|_{L^1(\R)} 
& \leq \frac{C_{\R,\Omega}}{\varepsilon}  \cdot \left\|\sum_{j=1}^N \eta_j\right\|_{L^1(\R^2)},\label{eq:A-eta-slice-gauss}
\end{align}
where $C_{\R,\Omega}:= |\Omega_x| $ where $\Omega_x = \{ x \in \R: (x,0) \in \Omega \} \subset \mathbb{R}$. 
\end{lemma}
\begin{proof}
    Formula (\ref{eqn:summed-amb-func-relation}) gives 
\begin{align*}
\left\|A-A^\eta\right\|_{L^{p,1}(\R^2)} =  \left\| \mathds{1}_\Omega \frac{\sum_{j=1}^N T \mathcal{F}(F_j-F_j^\eta)}{\sum_{j=1}^N \mathcal{A} \phi_j} \right\|_{L^{p,1}(\R^2)}\leq \frac{C_p}{\varepsilon}  \left\| \, T \sum_{j=1}^N    \mathcal{F}(F_j-F_j^\eta) \right\|_{L^\infty(\R^2)},
\end{align*}
where $T(\cdot)(x,y) = (\cdot)(y,-x)$. This can be further estimated by
\begin{align}%
\left\|A-A^\eta\right\|_{L^{p,1}(\R^2)} & \leq \frac{C_p}{\varepsilon}   \left\|\mathcal{F} \left( \sum_{j=1}^N \eta_j \right)\right\|_{L^\infty(\mathbb{R}^2)} \leq \frac{C_p}{\varepsilon}   \left\|\sum_{j=1}^N \eta_j\right\|_{{L^1}(\mathbb{R}^2)} , \nonumber
\end{align}
through employing the Hausdorff--Young inequality in the last step. Finally, (\ref{eq:A-eta-slice-gauss}) can be obtained as
\begin{align*}
\left\|A(0,\cdot)-A^\eta(0,\cdot)\right\|_{L^1(\R)} &\leq  \left\| \frac{\sum_{j=1}^N T \mathcal{F} (F_j -F_j^\eta)}{\sum_{j=1}^N \mathcal{A} \phi_j}(0,\cdot)\right\|_{L^1(\Omega \cap (\{ 0\} \times \R))},\nonumber\\
&\leq \frac{1}{\varepsilon} \left\| \sum_{j=1}^N \mathcal F \eta_j (\cdot,0)\right\|_{L^1(\Omega_x)}, \nonumber\\
&\leq \frac{C_{\R,\Omega}}{\varepsilon} \left\| \sum_{j=1}^N \mathcal{F} \eta_j (\cdot, 0) \right\|_{L^\infty(\R)},  \nonumber \\
&\leq \frac{C_{\R,\Omega}}{\varepsilon} \left\| \mathcal{F} \Big( \sum_{j=1}^N  \eta_j  \Big) \right\|_{L^\infty(\R^2)},  \nonumber \\
& \leq \frac{C_{\R,\Omega}}{\varepsilon}  \cdot \left\|\sum_{j=1}^N \eta_j\right\|_{L^1(\R^2)},\nonumber
\end{align*}
where the last inequality is an application of the Hausdorff--Young inequality.
\end{proof}

To bound the propagation of data noise in the reconstructions $\f$ and  $\feta$,  we first define
\begin{align*}
\g_c(x) &:= \int_{\R} A(x,y) e^{2 \pi i (x/2+c) y} \rd y,\\
\g^\eta_c(x) &:= \int_{\R} A^\eta(x,y) e^{2 \pi i (x/2+c) y} \rd y,
\end{align*}
for $c \in \R$ and note that their difference can be straightforwardly bounded by
\begin{align}
\left\|\g_c-\g_c^\eta\right\|_{L^p(\R)} &\leq \left( \int_\R \left| \int_\R \left(A(x,y)-A^\eta(x,y) \right) e^{2 \pi i (x/2+c) y} \rd y \right|^p \rd x \right)^{1/p},\nonumber\\
&\leq \left( \int_\R \left| \int_\R \left| A(x,y)-A^\eta(x,y) \right|  \rd y \right|^p \rd x \right)^{1/p},\nonumber\\
&\leq \int_\R \left( \int_\R \left| A(x,y)-A^\eta(x,y) \right|^p \rd x \right)^{1/p} \rd y = \left\|A-A^\eta\right\|_{L^{p,1}(\R^2)}, \label{eq:g-A}
\end{align}
where the last inequality is an application of Minkowski's integral inequality.

For reconstructing from $\g_c,\g_c^\eta$, it is necessary to fix a point $c \in \R$ such that $|\g_c(0)| \neq 0$ and $|\g_c^\eta(0)| \neq 0.$ Without loss of generality and for simpler notation, suppose that one can choose $c=0$ and denote $\g(x) := \g_0(x),$ $\g^\eta(x) := \geta_0(x).$ Then, the reconstructions are obtained by taking
\begin{align*}
\f(x) &= \frac{1}{\sqrt{\left|\g(0)\right|}} \g(x),\\
\feta(x) &= \frac{1}{\sqrt{\left|\g^\eta(0)\right|}} \g^\eta(x),
\end{align*}
which corresponds to setting the phase factor at $x=0$ to be equal to $1$, i.e., $\f(0)=|\f(0)|$ and $\feta(0)=|\feta(0)|$ . Note that with the choice above, $|\f(0)|=\sqrt{|\g(0)|}$ and $|\feta(0)|=\sqrt{|\g^\eta(0)|}$. 

The propagation of data noise on the reconstructions $\f$ and $\feta$ can be estimated as follows:
\begin{proposition}
Let $k := \sqrt{\left|\widetilde{u}(0)\right|}$ and $k^\eta := \sqrt{\left|\widetilde{u}^\eta(0)\right|}$.
Then,
\begin{align}
   \left\|\f-\feta\right\|_{L^p(\R)} %
   & \leq \left(\frac{\left\|A\right\|_{L^{p,1}(\R^2)}}{2 (k \, k^\eta)^{3/2}}  + \frac{1}{k^\eta}\right)  \frac{\max(C_{\R,\Omega}, C_{p})}{\varepsilon}   \left\|\sum_{j=1}^N \eta_j\right\|_{{L^1}(\R^2)}.
   \label{eq:rec-error-gauss}
\end{align}
\end{proposition}
\begin{proof}
We have
\begin{align*}
\left\|\f-\feta\right\|_{L^p(\R)} &\leq \left\|\frac{1}{k} \cdot \g - \frac{1}{k^\eta} \cdot \g\right\|_{L^p(\R)} + \left\|\frac{1}{k^\eta} \cdot \g - \frac{1}{k^\eta} \cdot \g^\eta \right\|_{L^p(\R)},\\
&\leq \left|\frac{1}{k}-\frac{1}{k^\eta}\right| \cdot \left\|\g\right\|_{L^p(\R)} + \frac{1}{k^\eta} \left\| \g -  \g^\eta \right\|_{L^p(\R)},\\
&\leq \frac{1}{2 (k \, k^\eta)^{3/2}} \left|(k^\eta)^2-k^2 \right|\cdot \left\|\g\right\|_{L^p(\R)} + \frac{1}{k^\eta} \left\| \g -  \g^\eta \right\|_{L^p(\R)},\\
&\leq \frac{1}{2 (k \, k^\eta)^{3/2}} \left|\g^\eta(0)-\g(0) \right|\cdot \left\|\g\right\|_{L^p(\R)} + \frac{1}{k^\eta} \left\| \g -  \g^\eta \right\|_{L^p(\R)},
\end{align*}
where the arithmetic-geometric mean inequality is applied in the third line (since $k, k^\eta>0$) and the reverse triangle inequality is used in the fourth line. 
Combining this with (\ref{eq:g-A}), we obtain
\[
\left\|\f-\feta\right\|_{L^p(\R)} \leq \frac{1}{2 (k \, k^\eta)^{3/2}} \left\|A(0,\cdot)-A^\eta(0,\cdot)\right\|_{L^1(\R)} \cdot \left\|A\right\|_{L^{p,1}(\R^2)} + \frac{1}{k^\eta} \left\|A-A^\eta\right\|_{L^{p,1}(\R^2)}. %
\]
Further estimating the RHS employing (\ref{eq:A-eta-gauss}) and (\ref{eq:A-eta-slice-gauss}) yields the desired bound.
\end{proof}

\subsubsection{Approximation error}

The error bound in~\eqref{eq:rec-error-gauss} becomes small as we increase $\varepsilon$. However, this comes  at the cost of increasing $\|f-\f\|$ since we only recover $\mathcal A f$ or its approximation on the bounded domain $\Omega=\bigcup_{j=1}^N \Omega_j$ decided by $\varepsilon$. The bound on the approximation error $\|f-\f\|$ will depend on the quantities
\begin{equation}
\left\|\mathcal{A} f \right\|_{L^{p,1}(\R^2 \backslash \Omega)} :=  \left\|\mathds{1}_{\R^2\backslash \Omega}\mathcal{A} f \right\|_{L^{p,1}(\R^2)}, \label{eq:A-concentration}
\end{equation}
and
\begin{equation}
\left\| \mathcal{A}f(0,\cdot) \right\|_{L^1((\R^2 \backslash \Omega) \cap (\left\{0\right\} \times \R))}.  \label{eq:A-slice-concentration}
\end{equation}
Note that $|\mathcal{A}f|=|V_f f|,$ so~\eqref{eq:A-concentration} quantifies the decay of both $f$ and its Fourier transform $\widehat{f}.$ Moreover, \eqref{eq:A-slice-concentration} gives the decay of $|\mathcal{F}(|f|^2)|$ since $|\mathcal{A}f(0,\cdot)| = |\mathcal{F}(|f|^2)|$.

We further note that since $A$ and $\mathcal{A}f$ agree on $\Omega$ and $A$ is identical to zero on $\R^2\backslash \Omega$, we have %
\[
\left\|A-\mathcal{A}f\right\|_{L^{p,1}(\R^2)} = \left\|A-\mathcal{A}f\right\|_{L^{p,1}(\R^2 \backslash \Omega)} = \left\|\mathcal{A} f \right\|_{L^{p,1}(\R^2 \backslash \Omega)},
\]
and, similar to (\ref{eq:g-A}), 
$$
\left\|u-\g\right\|_{L^p(\R)} \leq \left\|\mathcal{A}f-A\right\|_{L^{p,1}(\R^2)},
$$
where $u(x):=\int_\R \mathcal{A} f(x,y) e^{\pi i x y} \rd y$. Note that $u(0)=|f(0)|^2>0$ and define  $k^{\mathrm{true}} := \sqrt{u(0)}$. As before, we take $k=\sqrt{|\g(0)|}$ and set the phase factors of $f$ and $\f$ to $1$ at $x=0$. With this, we obtain the following lemma.
\begin{lemma}
Let $f$ be the true solution and let $\widetilde{f}$ be the reconstruction from $F_1,\dots,F_N$ by formula (\ref{eqn:summed-amb-func-relation}). Then,
\begin{equation}
\left\|f-\f\right\|_{L^p(\R)}  \leq \frac{1}{k^{\mathrm{true}} } \left\|\mathcal{A} f \right\|_{L^{p,1}(\R^2 \backslash \Omega)} + \frac{\left\|A\right\|_{L^{p,1}(\R^2)}}{2 (k^{\mathrm{true}} \, k)^{3/2}} \left\| \mathcal{A}f(0,\cdot) \right\|_{L^1((\R^2 \backslash  \Omega) \cap (\left\{0\right\} \times \R))} .\label{eq:appr-error}
\end{equation}
\end{lemma}
\begin{proof}
We first use the triangle inequality and split the approximation error into two terms. Recalling that $$\widetilde{u}(x) = \int_\R A(x,y) e^{\pi i x y } \rd y \quad  \text{and} \quad \f(x)=\widetilde{u}(x)/k\,,$$ 
we obtain
\begin{align*}
\left\|f-\f\right\|_{L^p(\R)} &\leq \left\|f-\frac{\g}{k^{\mathrm{true}} } \right\|_{L^p(\R)} +  \left\|\frac{\g}{k^{\mathrm{true}} } - \f \right\|_{L^p(\R)} \nonumber \\
&= \left\|\frac{u}{k^{\mathrm{true}} }-\frac{\g}{k^{\mathrm{true}} } \right\|_{L^p(\R)} +  \left\|\frac{\g}{k^{\mathrm{true}} } - \f \right\|_{L^p(\R)} \nonumber \\
&\leq \frac{1}{k^{\mathrm{true}} } \left\|\mathcal{A} f \right\|_{L^{p,1}(\R^2 \backslash \Omega)}  + \left| \frac{1}{k^{\mathrm{true}} } - \frac{1}{k}\right| \cdot \left\| \widetilde{u}\right\|_{L^p(\R)} \nonumber \\
&\leq \frac{1}{k^{\mathrm{true}} } \left\|\mathcal{A} f \right\|_{L^{p,1}(\R^2 \backslash \Omega)}  + \left| \frac{1}{k^{\mathrm{true}} } - \frac{1}{k}\right| \cdot \left\|A\right\|_{L^{p,1}(\R^2)}.
\end{align*}
Similar to the derivations of the noise propagation term, we can further bound this by
\begin{align*}
\left\|f-\f\right\|_{L^p(\R)} &\leq \frac{1}{k^{\mathrm{true}} } \left\|\mathcal{A} f \right\|_{L^{p,1}(\R^2 \backslash \Omega)} + \frac{1}{2 (k^{\mathrm{true}} \, k)^{3/2}} \left|u(0)-\g(0)\right| \cdot \left\|A\right\|_{L^{p,1}(\R^2)} \nonumber \\
& \leq \frac{1}{k^{\mathrm{true}} } \left\|\mathcal{A} f \right\|_{L^{p,1}(\R^2 \backslash \Omega)} + \frac{1}{2 (k^{\mathrm{true}} \, k)^{3/2}} \left\| \mathcal{A}f(0,\cdot)-A(0,\cdot) \right\|_{L^1(\R)} \cdot \left\|A\right\|_{L^{p,1}(\R^2)} \nonumber \\
& \leq \frac{1}{k^{\mathrm{true}} } \left\|\mathcal{A} f \right\|_{L^{p,1}(\R^2 \backslash \Omega)} + \frac{1}{2 (k^{\mathrm{true}} \, k)^{3/2}} \left\| \mathcal{A}f(0,\cdot) \right\|_{L^1((\R^2 \backslash  \Omega) \cap (\left\{0\right\} \times \R))} \cdot \left\|A\right\|_{L^{p,1}(\R^2)}.%
\end{align*}
\end{proof}

\subsubsection{Total error}
We are now in a position to combine the estimates (\ref{eq:rec-error-gauss}) and (\ref{eq:appr-error}) on propagation of data noise and approximation error to a total error estimate. 

\begin{proposition}
For some constant $C>0$ depending on $\|A\|_{L^{p,1}(\R^2)}, k, k^\eta, k^{\mathrm{true}}, C_{\R,\Omega}$ and $ C_{p},$ we have
\begin{align}
\left\|f-\feta\right\|_{L^p(\R)} & \leq \frac{1}{k^{\mathrm{true}} } \left\|\mathcal{A} f \right\|_{L^{p,1}(\R^2 \backslash \Omega)} + \frac{1}{2 (k^{\mathrm{true}} \, k)^{3/2}} \left\| \mathcal{A}f(0,\cdot) \right\|_{L^1((\R^2 \backslash  \Omega) \cap (\left\{0\right\} \times \R))} \cdot \left\|A\right\|_{L^{p,1}(\R^2)} \nonumber \\
&\quad  + \left(\frac{1}{2 (k \, k^\eta)^{3/2}} \cdot \left\|A\right\|_{L^{p,1}(\R^2)} + \frac{1}{k^\eta}\right) \cdot \frac{1}{\varepsilon} \max(C_{\R,\Omega}, C_{p})  \left\| \sum_{j=1}^N \eta_j\right\|_{{L^1}(\R^2)}    \nonumber\\
 & \leq C \cdot \left( \left\|\mathcal{A} f \right\|_{L^{p,1}(\R^2 \backslash \Omega)} + \left\| \mathcal{A}f(0,\cdot) \right\|_{L^1((\R^2 \backslash  \Omega) \cap (\left\{0\right\} \times \R))}+ \frac{1}{\varepsilon}  \left\|\sum_{j=1}^N \eta_j\right\|_{{L^1}(\R^2)} \right).\label{eq:general_total_error_2}
\end{align}
\end{proposition}

\begin{remark}
There are three terms on the RHS of~\eqref{eq:general_total_error_2}. The first two terms depend on $\Omega = \cup_{j=1}^N \Omega_j$, which is determined by the choice of $\varepsilon$. As $\varepsilon$ increases, so do the first two terms. Conversely, the last term is inversely proportional to $\varepsilon$. Therefore, one must balance these terms when selecting an appropriate $\varepsilon$.

Now, consider a fixed $\varepsilon$. If we only use a single window function, i.e., $N=1$, then the set $\Omega$ can be significantly smaller compared to the case with multiple windows; see Figs.~\ref{fig:frac_gauss_show} and~\ref{fig:hermite_show} for illustrations. This again sheds lights on the fact that we achieve better stability by using more windows. In this way, we can significantly reduce the magnitude of the first two terms in~\eqref{eq:general_total_error_2} while keeping the last term constant.

\end{remark}

\subsection{Stability of Algorithm~\ref{alg2}}\label{sec:stab-alg2}

In Algorithm~\ref{alg2}, the reconstruction is done separately on each $\Omega_j$, using $F_j$, $F_j^\eta$ for computing $A$, $A^\eta$, respectively. {Note that in Algorithm~\ref{alg2}, the sets  $\Omega_j$'s are constructed to be pairwise disjoint.}
Again, $A$ and $A^\eta$ will be set to zero outside of $\Omega$ and we obtain the following error bound for the data noise propagation: 
\begin{lemma}\label{lem:alg2}
Let $A, A^\eta$ be the reconstructions obtained from $F_j,$ $F_j^\eta,$ respectively, $j=1,\dots, N$, and for some $\varepsilon > 0$, let $\Omega_j$, {pairwise disjoint,} be chosen such that $\left|\mathcal{A}h_{n_j} (x,y) \right|> \varepsilon$ for all $(x,y) \in \Omega_j$ for all $j \in \{1,\dots, N\}$. Then,
\begin{eqnarray*}
\left\|A-A^\eta\right\|_{L^{p,1}(\R^2)}  &\leq& \frac{C_p}{\varepsilon}  \max_{j=1,\dots,N} \left\|\eta_j\right\|_{{L^1}(\R^2)},\\
\left\|A(0,\cdot)-A^\eta(0,\cdot)\right\|_{L^1(\R)} 
&\leq&  \frac{C_{\R,\Omega}}{\varepsilon}  \cdot \max_{j=1,\dots,N} \left\|\eta_j\right\|_{L^1(\R^2)}.
\end{eqnarray*}
\end{lemma}
\begin{proof}
By the reconstruction formula,
\[
\left\|A-A^\eta\right\|_{L^{p,1}(\R^2)} \leq \sum_{j=1}^N \left\| \mathds{1}_{\Omega_j}\frac{T \mathcal{F}(F_j-F_j^\eta)}{\mathcal{A} h_{n_j}} \right\|_{L^{p,1}(\R^2)}\leq \sum_{j=1}^N C_{j,p} \cdot \left\| \mathds{1}_{\Omega_j} \frac{T \mathcal{F}(F_j-F_j^\eta)}{\mathcal{A} h_{n_j}} \right\|_{L^\infty(\R^2)},
\]
where {$C_{j,p}:= \|\mathds{1}_{\Omega_j}\|_{L^{p,1}(\R^2)} < \infty$} and instead of (\ref{eq:A-eta-gauss}), we now have the similar bound
\begin{align*}%
\left\|A-A^\eta\right\|_{L^{p,1}(\R^2)} & \leq \sum_{j=1}^N \frac{C_{j,p}}{\inf_{\Omega_j} |\mathcal{A} h_{n_j}|} \cdot \left\| \mathcal{F} \eta_j\right\|_{L^\infty(\mathbb{R}^2)} \nonumber \\
&\leq \frac{1}{\varepsilon} \sum_{j=1}^N C_{j,p} \cdot \left\|\eta_j\right\|_{{L^1}(\R^2)}  \leq \frac{C_p}{\varepsilon}  \max_{j=1,\dots,N} \left\|\eta_j\right\|_{{L^1}(\R^2)},
\end{align*}
where we have used $\sum_{j=1}^N C_{j,p} = C_p$ (since $\Omega_1,\dots, \Omega_N$ are pairwise disjoint and $\cup_{j=1}^N \Omega_j = \Omega$), and the Hausdorff--Young inequality.  
Similar to (\ref{eq:A-eta-slice-gauss}), we can further derive
\begin{align*}
\left\|A(0,\cdot)-A^\eta(0,\cdot)\right\|_{L^1(\R)} &\leq \sum_{j=1}^N \left\| \frac{T \mathcal{F} (F_j -F_j^\eta)}{\mathcal{A} h_{n_j}}(0,\cdot)\right\|_{L^1(\Omega_j \cap (\{ 0\} \times \R))}\nonumber\\
&\leq \frac{1}{\varepsilon}\sum_{j=1}^N \left\| \mathcal F \eta_j (\cdot,0)\right\|_{L^1(\Omega_{j,x})}\nonumber\\
&\leq \frac{1}{\varepsilon}\sum_{j=1}^N \left\| \mathcal{F} \eta_j (\cdot, 0) \right\|_{L^\infty(\R)} \cdot \left| \Omega_{j,x}\right| \nonumber \\
&\leq \frac{1}{\varepsilon} \sum_{j=1}^N \left\|\eta_j\right\|_{L^1(\R^2)} \cdot \left| \Omega_{j,x}\right|\nonumber\\
& \leq \frac{C_{\R,\Omega}}{\varepsilon}  \cdot \max_{j=1,\dots,N} \left\|\eta_j\right\|_{L^1(\R^2)}, %
\end{align*}
where $\Omega_{j,x} = \{ x \in \R: (x,0) \in \Omega_j \} \subset \mathbb{R}$.
\end{proof}

With Lemma~\ref{lem:alg2}, the propagation of data noise can be estimated by
\begin{align*}
   \left\|\f-\feta\right\|_{L^p(\R)} %
   & \leq \left(\frac{1}{2 (k \, k^\eta)^{3/2}} \cdot \left\|A\right\|_{L^{p,1}(\R^2)} + \frac{1}{k^\eta}\right) \cdot \frac{1}{\varepsilon} \max(C_{\R,\Omega}, C_{p}) \max_{j=1,\dots,N} \left\|\eta_j\right\|_{{L^1}(\R^2)}.
\end{align*}
The approximation error is the same as in the previous case (cf.~(\ref{eq:appr-error})), so that the total error can be bounded by
\begin{align}
\left\|f-\feta\right\|_{L^p(\R)} & \leq \frac{1}{k^{\mathrm{true}} } \left\|\mathcal{A} f \right\|_{L^{p,1}(\R^2 \backslash \Omega)} + \frac{1}{2 (k^{\mathrm{true}} \, k)^{3/2}} \left\| \mathcal{A}f(0,\cdot) \right\|_{L^1((\R^2 \backslash  \Omega) \cap (\left\{0\right\} \times \R))} \cdot \left\|A\right\|_{L^{p,1}(\R^2)} \nonumber \\
&\quad + \left(\frac{1}{2 (k \, k^\eta)^{3/2}} \cdot \left\|A\right\|_{L^{p,1}(\R^2)} + \frac{1}{k^\eta}\right) \cdot \frac{1}{\varepsilon} \max(C_{\R,\Omega}, C_{p}) \max_{j=1,\dots, N} \left\|\eta_j\right\|_{{L^1}(\R^2)}  \nonumber\\
 & \leq C \cdot \left( \left\|\mathcal{A} f \right\|_{L^{p,1}(\R^2 \backslash \Omega)} + \left\| \mathcal{A}f(0,\cdot) \right\|_{L^1((\R^2 \backslash  \Omega) \cap (\left\{0\right\} \times \R))}+ \frac{1}{\varepsilon} \max_{j=1,\dots,N} \left\|\eta_j\right\|_{{L^1}(\R^2)}  \right),\label{eq:general_total_error}
\end{align}
for some constant $C>0$ depending on $\|A\|_{L^{p,1}(\R^2)}, k, k^\eta, k^{\mathrm{true}}, C_{\R,\Omega}$ and $ C_{p}.$

The third term on the RHS is the only difference between~\eqref{eq:general_total_error_2} and~\eqref{eq:general_total_error}. {The one in~\eqref{eq:general_total_error} depends on the maximum noise power, while the one in~\eqref{eq:general_total_error_2}  involves the sum of the noise for different window functions.} Their differences result from the distinctions between Algorithms~\ref{alg1} and~\ref{alg2} regarding how the multi-window measurements are utilized in the STFT phase retrieval.

\section{Numerical examples}\label{sec:numerics}
In this section, we present a few 1D phase retrieval examples using the explicit formula~\eqref{eqn:amb-func-relation}-\eqref{eqn:f-from-Af}, given the phaseless measurement $|V_g f|$ and the ambiguity function $\mathcal{A} g$ of the window function $g$ or $V_g g$.

\subsection{Discretization}
Recall the continuous STFT transform given in Equation~\eqref{eq:gabor}. If we know the function $f$ is compactly supported on a finite interval of $\R$, e.g., $[-T,T]$ for some $T>0$, then we can perform the integration on $[-T,T]$:
\begin{equation}\label{eq:vgf}
V_g f(x,y):=  \int_{-T}^{T} f(t) \overline{g(t-x)} e^{-2\pi i t y} \,  \rd t.
\end{equation}
Without loss of generality, we only consider periodic window functions $g(t)$ on $[-T,T]$, i.e., $g(-t) = g(2T-t)$. All Hermite and Gauss windows used in the numerical examples are periodized.

For the numerical integration, consider an equidistant discretization in time such that $f_l := f(x_l)$, $x_l = -T +   l \Delta t$, $0\leq l \leq L-1$ and $L  = 2T/\Delta t$. We compute the STFT at $K$ equidistant time-domain shifts $x_0,\ldots,x_k,\ldots,x_{K-1}$, based on the uniform spacing $a\Delta t$ with $a\in \mathbb{N}^+$ and $K =L/a$.
Similarly, we discretize the frequency domain interval $\left[-\frac{1}{2\Delta t},\frac{1}{2 \Delta t}\right]$ uniformly by taking $M$ points, i.e., $y_m =  \frac{1}{\Delta t} \left( - \frac{1}{2} +  \frac{m}{M} \right)$ for $0\leq m < M$. Note that $1/(2 \Delta t)$ is the Nyquist frequency. Using the left Riemann sum, \eqref{eq:vgf} can be approximated by
\begin{equation}\label{eq:dgt}
c_{\left(m, k\right)}= \Delta t \sum_{l=0}^{L-1} f_l\, \overline{g_{l-k a + L/2 } } \, e^{-2\pi i   (-T + l \Delta t)   \frac{1}{\Delta t} \left(-\frac{1}{2}  +  \frac{ m}{M}  \right)} ,
\end{equation}
where $  m= 0,\ldots,M-1,$ and $ k=0,\ldots,K-1,$
with $$c_{\left(m, k\right)} \approx V_g f\left( -T + k a \Delta t,  \frac{1}{\Delta t} \left( - \frac{1}{2} +  \frac{m}{M} \right) \right)$$ 
being the $(m+1, k+1)$-th entry of the coefficient matrix $c$.   The index $l-ka + L/2$ in~\eqref{eq:dgt} is computed modulo $L$ since we assume that the window function $g$ is periodic. That is, 
$$
g_{-s} = g(- T -  s \Delta t) = g(T-s \Delta t) = g_{(2T-s \Delta t)/\Delta t} = g_{L-s}, \quad \forall \, s = 0,\ldots, L-1.
$$
There are many fast ways of implementing~\eqref{eq:dgt}; see the toolbox~\cite{sondergaard2012linear} for an example. One should note that if $\Delta t = 1$, \eqref{eq:dgt} reduces to the so-called Discrete Short-Time Fourier Transform (Discrete STFT).

Since the output of our proposed algorithms recovers the original signal up to a global phase, we define the following misfit function 
 $d(\, \cdot \,, \, \cdot \,)$ that measures the relative error between two signals $\text{f}_1$ and $\text{f}_2$ up to a global phase factor: %
\begin{equation}\label{eq:obj}
    d(\text{f}_1,\text{f}_2 ) = \min_{\theta\in [0,2\pi) }
    \frac{ \|\text{f}_1\ - e^{i\theta} \text{f}_2\|_{\ell^2} }{\|\text{f}_1\|_{\ell^2}}.
\end{equation}
The angle $\theta^*$ that minimizes~\eqref{eq:obj} is also the missing global phase angle.

\begin{figure}
    \centering
    \subfloat[Chirp signal $f(t)$]{\includegraphics[width = 0.4\textwidth]{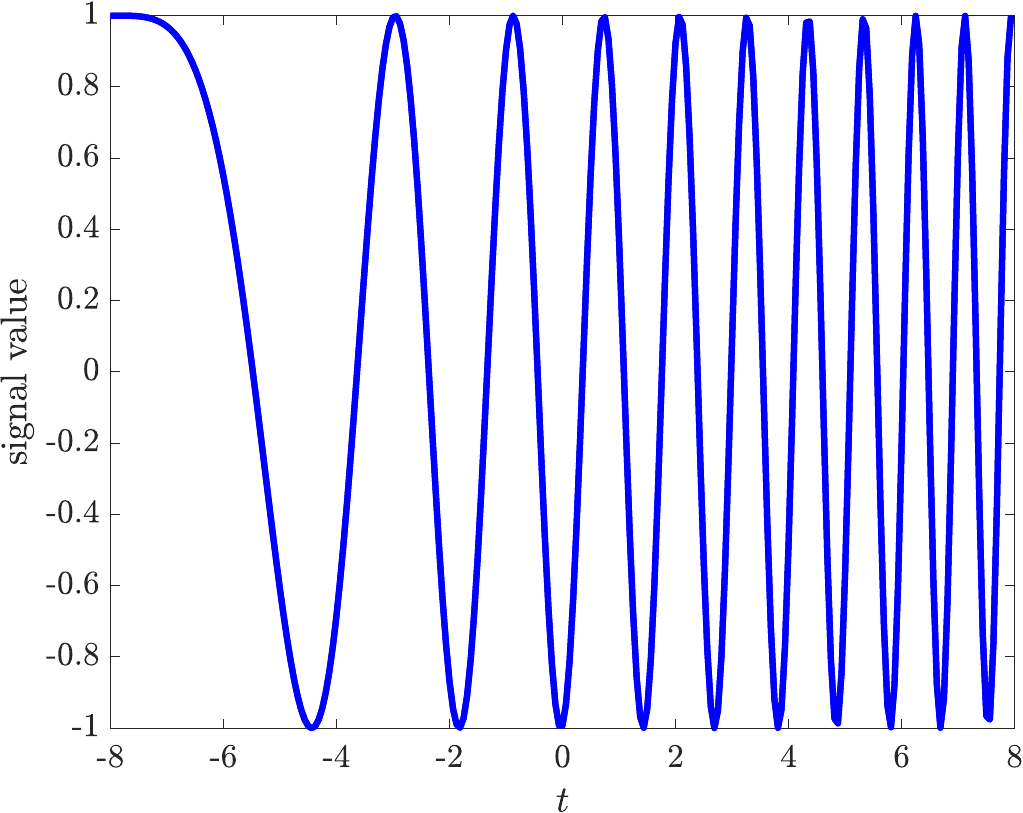}\label{fig:chip}}
    \hspace{0.5cm}
    \subfloat[$\log |\mathcal{A} f(x,y)|$]{\includegraphics[width = 0.4\textwidth]{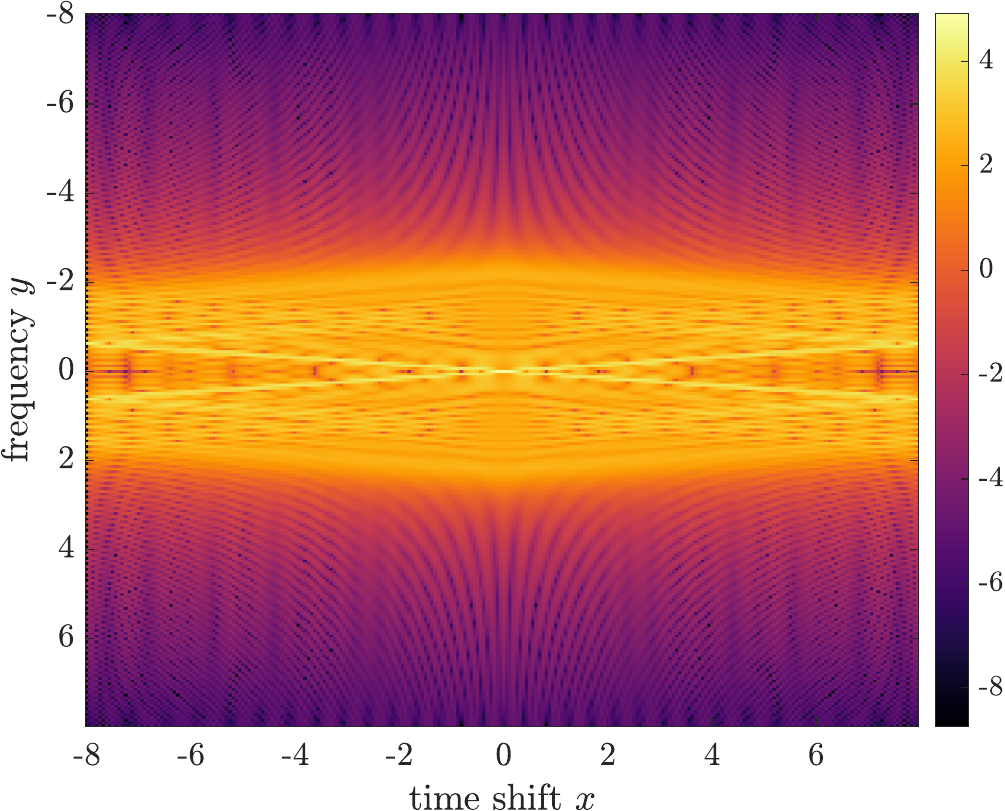}\label{fig:chip_vff}}
    \caption{(a) The ground truth chirp signal $f(t)$; (b) The absolute value of the ambiguity function of $f$ plotted in the log scale.}
    \label{fig:chirp_show}
\end{figure}

\begin{figure}
    \centering
    \subfloat[Recovery with $40$ FRFT dilated Gauss windows]{\includegraphics[width = 0.48\textwidth]{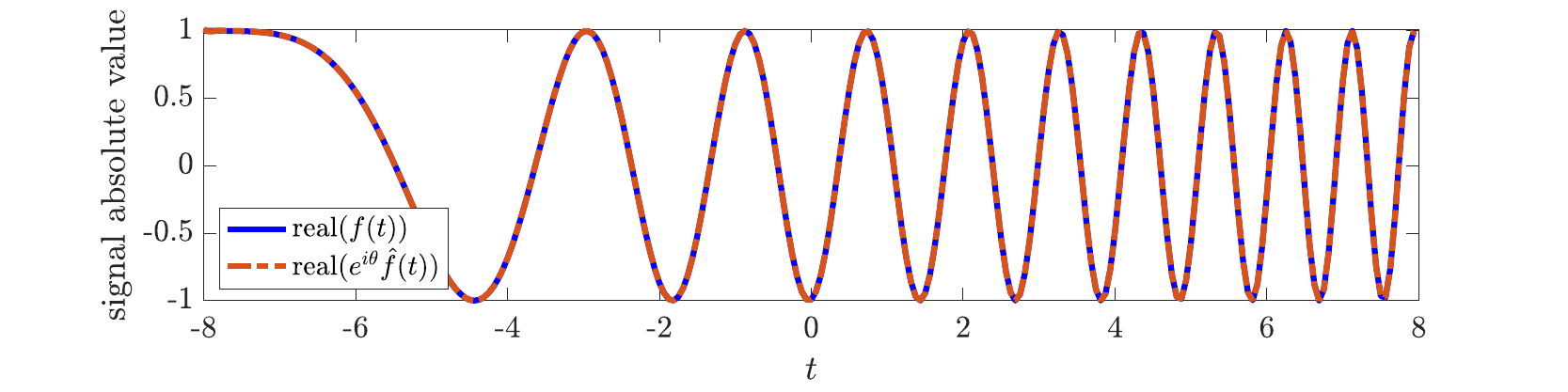}\label{fig:Gauss15-40}}
    \hspace{0.03\textwidth}
    \subfloat[Recovery with a standard Gauss  window, $\varepsilon = 10^{-3}$]{\includegraphics[width = 0.48\textwidth]{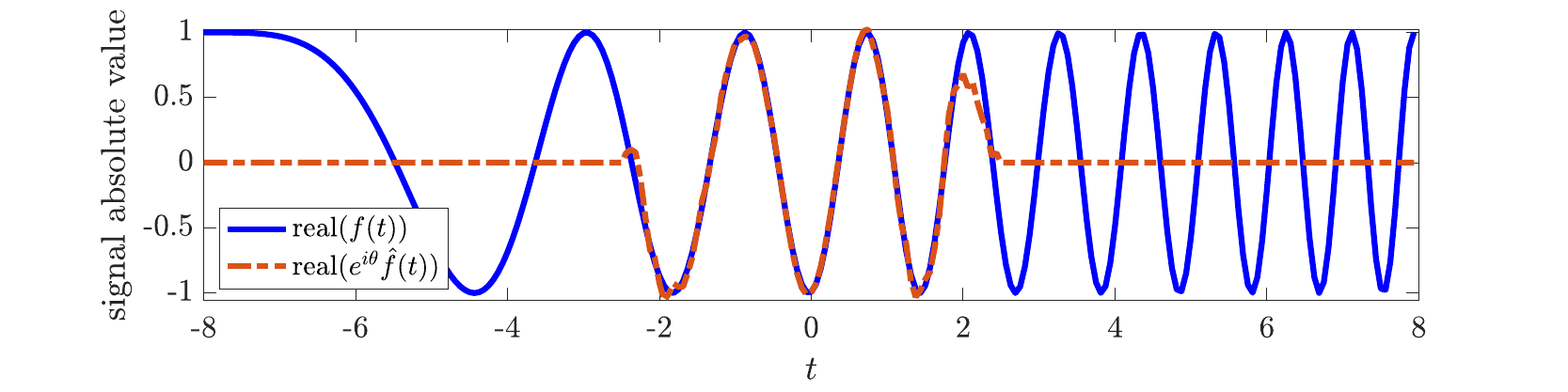}\label{fig:Gauss1}}\\
    \subfloat[Recovery with a standard Gauss  window, $\varepsilon = 10^{-6}$]{\includegraphics[width = 0.48\textwidth]{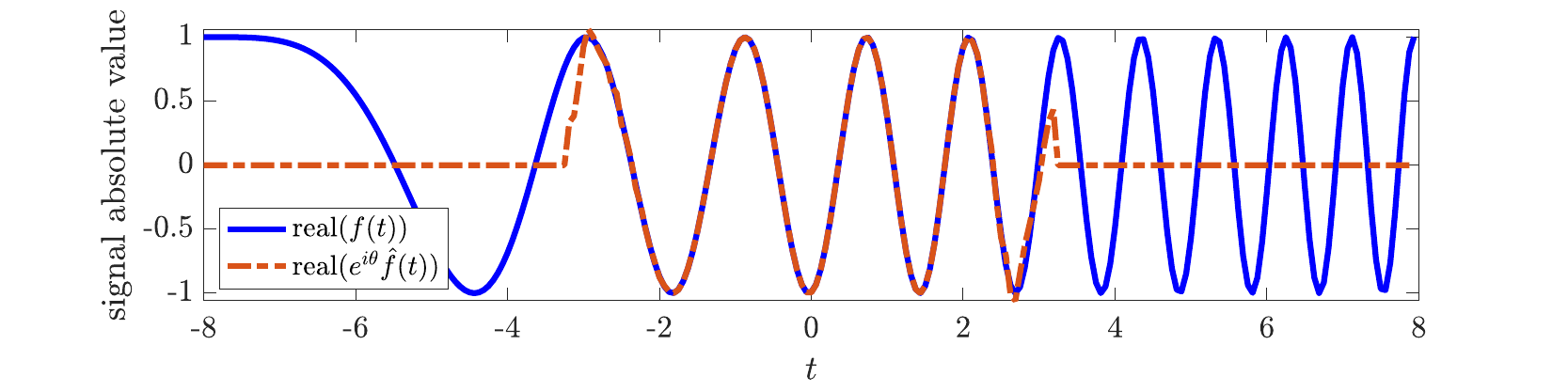}\label{fig:Gauss1e-6}}
    \hspace{0.03\textwidth}
    \subfloat[Recovery with a standard Gauss  window, $\varepsilon = 10^{-9}$]{\includegraphics[width = 0.48\textwidth]{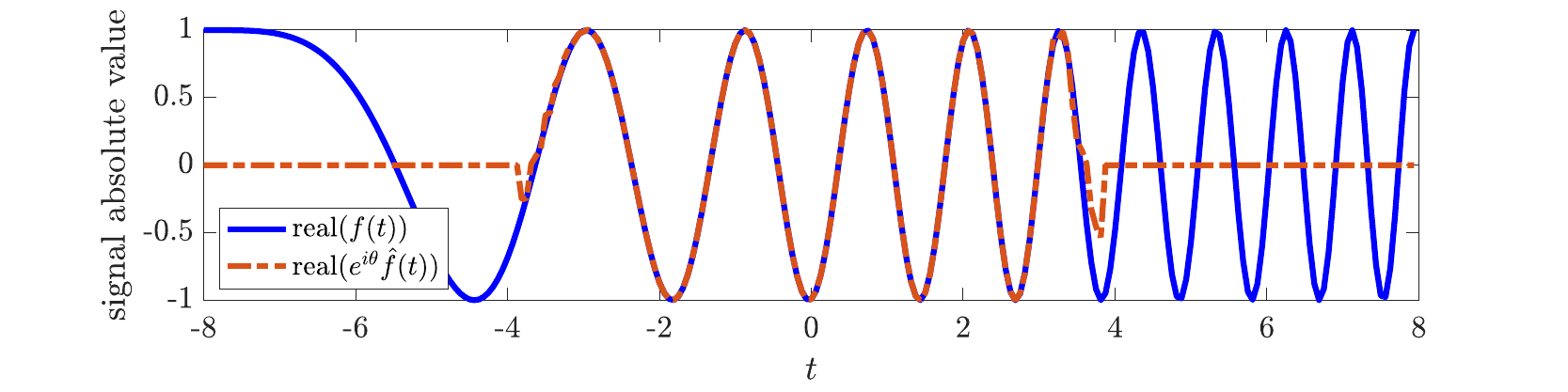}\label{fig:Gauss1e-9}}
    \caption{(a) Recovery using $40$ {fractional Fourier transformed dilated Gauss windows} at various angles; (b)-(d) Recovery using a single standard Gauss window with the cut-off $\varepsilon = 10^{-3}, 10^{-6}, 10^{-9}$, respectively. The ground truth is plotted in blue. }
    
    \label{fig:chirp_Gauss}
\end{figure}

\begin{figure}
    \centering
    \subfloat[Using $6$ Hermite windows picked uniformly]{\includegraphics[width = 0.50\textwidth]{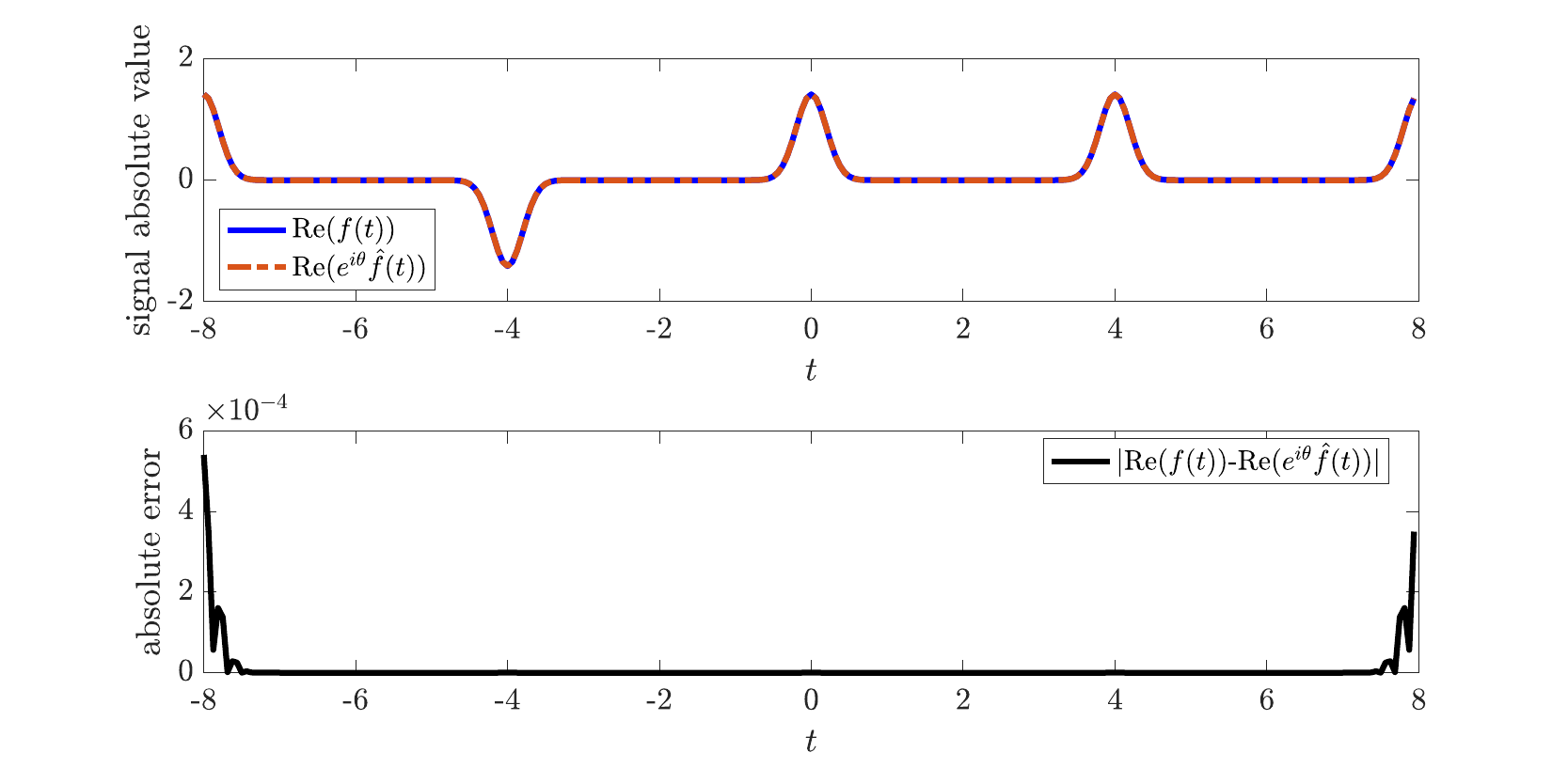}\label{fig:Hermite_uni}}
    \subfloat[Using $6$ Hermite windows picked randomly]{\includegraphics[width = 0.50\textwidth]{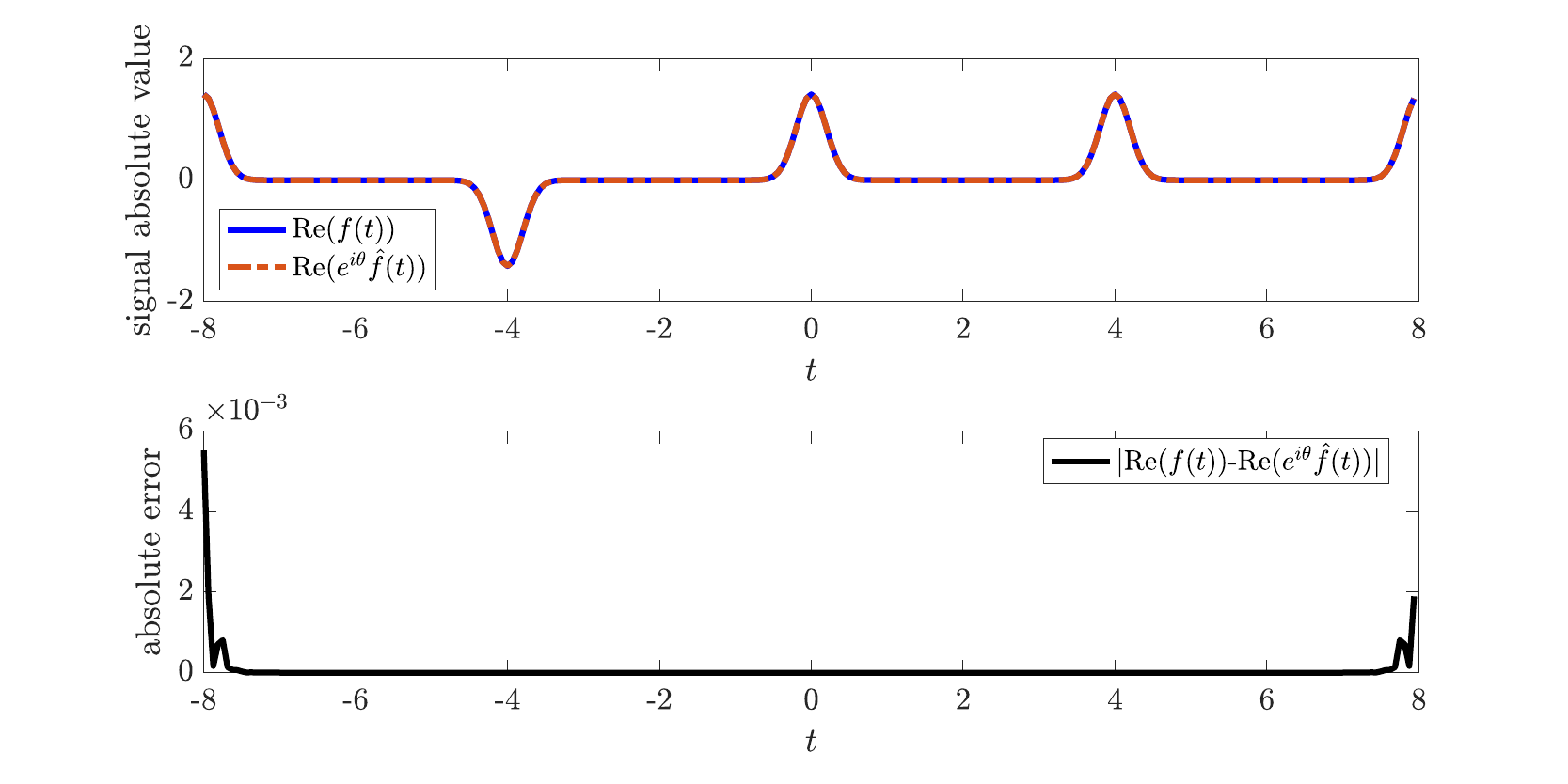}\label{fig:Hermite_rand}}
    \caption{(a) Recovery using $6$ Hermite windows $h_n$, where $n=0,10,20,30,40,50$; (b) Recovery using $6$ Hermite windows $h_n$ with the index $n$  randomly selected between $0$ and $50$, resulting in a set of indices $\{18,20,22,28,36,47\}$. The ground truth is shown in blue, the reconstructions obtained using Algorithm~\ref{alg2} are depicted as red dashed curves, and the absolute errors in the real components are illustrated with black solid lines.}
    
    \label{fig:chirp_Hermite}
\end{figure}

\subsection{Synthetic examples}
In this subsection, we will present two STFT phase retrieval examples using the proposed multi-window~\Cref{alg1,alg2}.
\subsubsection{Chirp signal recovery}
As our first inversion example, we consider the chirp signal shown in~Fig.~\ref{fig:chip}. As illustrated in~Fig.~\ref{fig:chip_vff}, its ambiguity function has extensive support in the time-frequency domain, making it challenging to reconstruct using a single window of the standard Gauss function whose ambiguity function has a very small essential support as seen in~Fig.~\ref{fig:frac_gauss_show_1}. We set $\varepsilon = 10^{-3}$ as a cut-off value, as used in~\Cref{alg1,alg2} to prevent the instability from dividing the phaseless measurements by a small value.
\begin{figure}
    \centering
    \subfloat[Gaussian mixture signal $f(t)$]{\includegraphics[width = 0.4\textwidth]{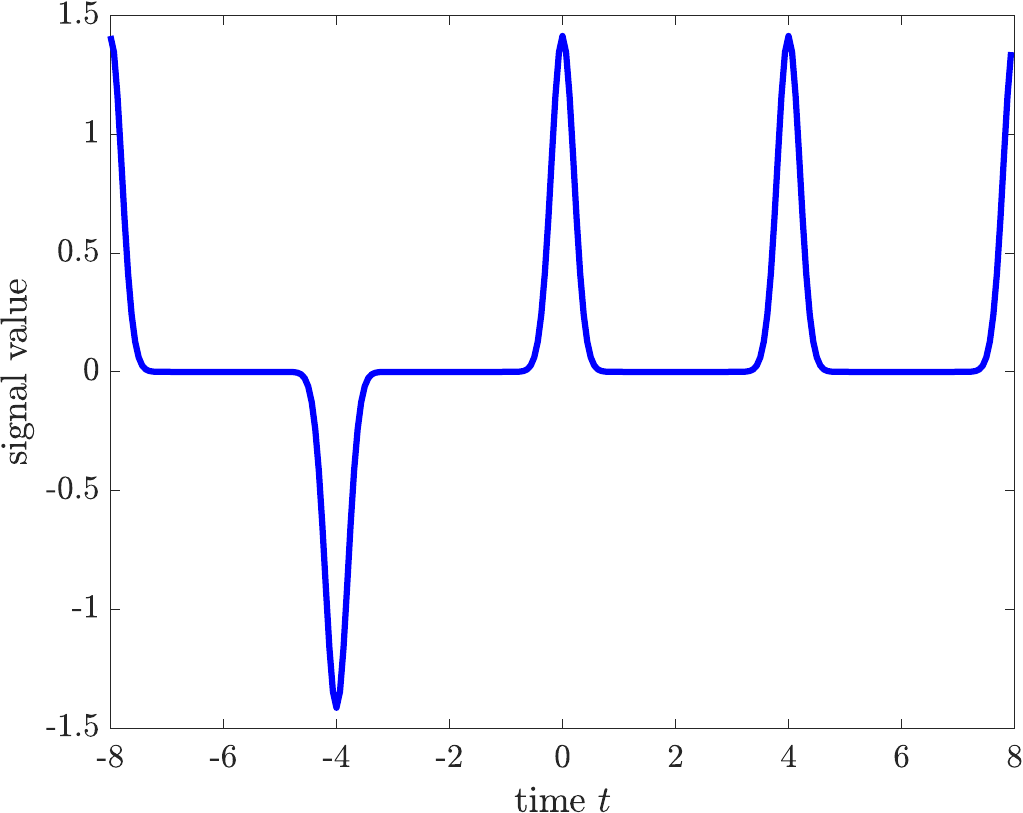}\label{fig:GM}}
    \hspace{0.5cm}
    \subfloat[$\log |\mathcal{A} f(x,y)|$]{\includegraphics[width = 0.4\textwidth]{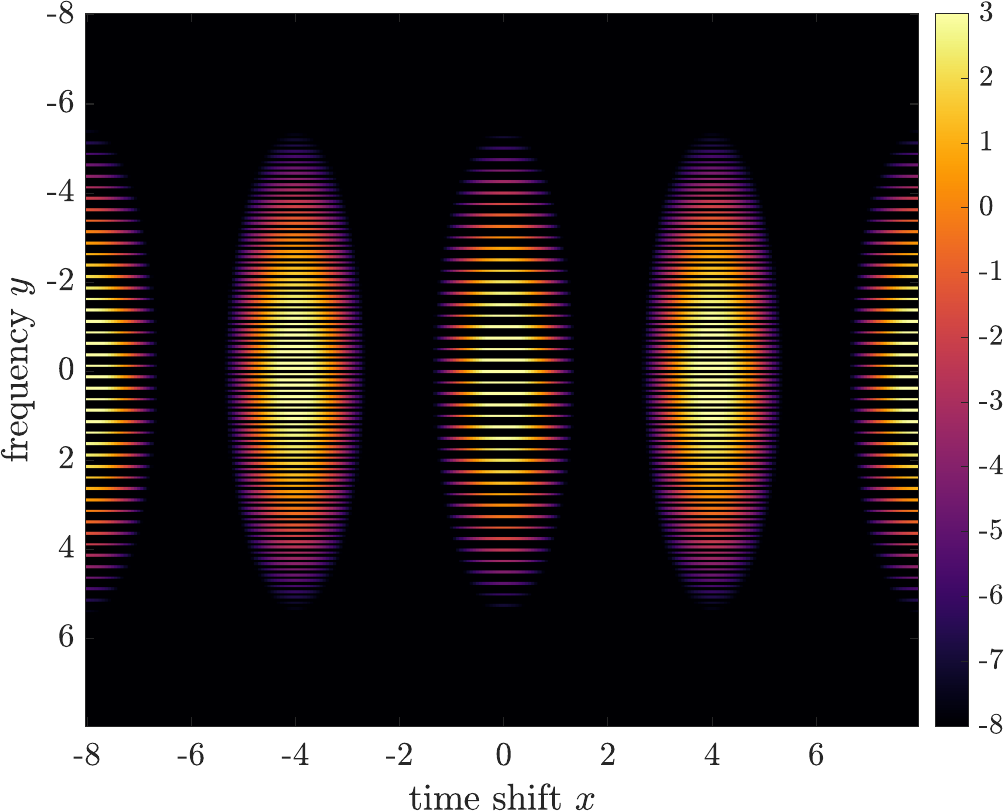}\label{fig:GM_vff}}
    \caption{Multi-modal function example: (a) The ground truth $f(t)$; (b) The absolute value of the ambiguity function of $f$ plotted in the log scale.}
    \label{fig:GM_show}
\end{figure}

We first use $40$ {fractional Fourier transformed} dilated Gaussian windows to perform the reconstruction following~\Cref{alg1}, i.e., $\mathcal{F}_{\alpha_j} \varphi^a$ with $a = 15$ and $\alpha_j = (j-1)\pi/40$, $j=1,\ldots, 40$. The relative error based on formula~\eqref{eq:obj}  between the ground truth and the reconstructed signal is $0.0084$.
We also use the standard Gauss function $\varphi^1$ as a single window function for comparison. The relative error is $0.7928$. After finding the global phase factor, we plot the reconstructions under different measuring window schemes in~Fig.~\ref{fig:chirp_Gauss}. More measuring windows yield a much more extensive joint coverage of the time-frequency domain, resulting in better signal reconstruction. The ambiguity function corresponding to the standard Gauss window has a very small essential support, as defined by the set $S_n^\varepsilon$ (see Fig.~\ref{fig:frac_gauss_show_1}), in contrast to the ambiguity function of the chirp signal, illustrated in Fig.~\ref{fig:chip_vff}. Due to this limited coverage, reconstruction with the standard Gauss window only partially recovers the original signal for a cut-off value of $\varepsilon = 10^{-3}$ (see Fig.~\ref{fig:Gauss1}). As $\varepsilon$ is decreased to $10^{-6}$ and $10^{-9}$, more parts of the signal are recovered, as shown in Figs.~\ref{fig:Gauss1e-6}--\ref{fig:Gauss1e-9}. However, the reconstruction remains incomplete even when $\varepsilon$ is the machine precision due to the inherently limited support of the ambiguity function associated with the standard Gauss window.

Next, we use multiple Hermite windows to perform the recovery. As demonstrated in Fig.~\ref{fig:hermite_show}, the essential support $S_n^\varepsilon $ of the ambiguity functions for Hermite functions of degree $ n $ does not vary significantly between lower and moderate degrees, such as $n=1$ and $n=5$. However, as the degree $n$ increases, specific circular gaps may emerge in the set $S_n^\varepsilon$, particularly at higher degrees (e.g., $n=10$). These gaps motivate the need for additional windows to ensure sufficient coverage of the joint support. We propose selecting Hermite windows with degrees spaced at regular intervals, such as every 5 degrees (e.g., $ n=0, 5, 10, \dots $). Intuitively, the degrees should be chosen to balance sufficient coverage of the joint support $\cup_{n \in I} S_n^\varepsilon$ and computational efficiency. A practical approach is to numerically evaluate the set $ \cup_{n \in I} S_n^\varepsilon $ for the chosen degrees $I$ to verify that the joint support is appropriately large. Alternatively, one could randomly select degree indices $I$ within the range $0$ (corresponding to the standard Gauss window) to a maximum degree $L$ and evaluate $\cup_{n \in I} S_n^\varepsilon$ to check for sufficient coverage. The random sampling method can sometimes reveal nontrivial degree combinations that improve reconstruction quality.

Consider $6$ Hermite function $h_n$ with $n\in \{0,10,20,30,40,50\}$. We then follow~\Cref{alg2} to perform the inversion. The misfit between the reconstructed signal and the ground truth measured by~\eqref{eq:obj} is $0.0025$. On the other hand, we want to test the idea of randomly choosing window functions. We draw
$6$  windows uniformly from the zeroth to the fiftieth degree Hermite functions, and use them to generate the phaseless data and perform the STFT phase retrieval following~\Cref{alg2}. We repeat this procedure $100$ times to study the performance of this inversion algorithm. The mean relative error  measured by~\eqref{eq:obj} between the  ground truth and the  reconstructed signal is $0.0298$. In particular, $90\%$ of the randomly picked window sets yield signal reconstructions with relative error less than $0.0320$. We also plot the reconstruction results from the uniformly chosen windows and one realization of the randomly selected windows in~Fig.~\ref{fig:chirp_Hermite}.

\subsubsection{Recovery of a multi-modal function with noisy measurements}

\begin{figure}
    \centering
    \includegraphics[trim={4.5cm 0 4.5cm 0},width = 1.0\textwidth,clip]{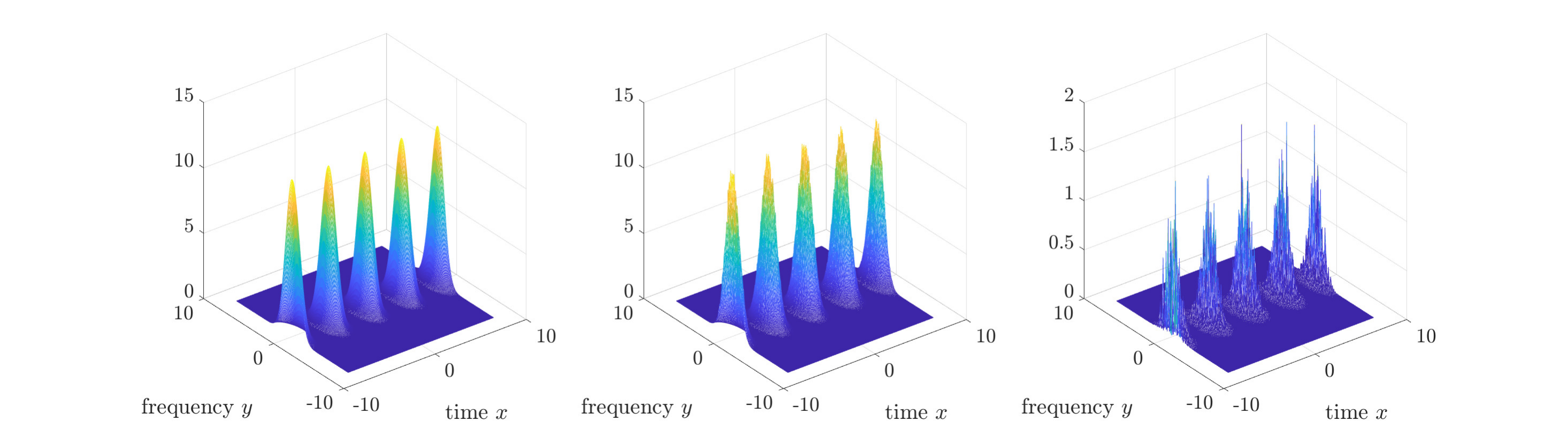}
    \caption{Multi-modal function example: the noise-free measurement of $|V_g f|$ for $g=\varphi^1$ (left), the noisy data used in phase retrieval (middle) and the absolute error between the noise-free measurement and the noisy one (right).}
    \label{fig:noise}
\end{figure}

We consider a multi-modal function plotted in~Fig.~\ref{fig:GM}, a linear combination of time-shifts of $\varphi^4$, the dilated Gauss function. The absolute value of its ambiguity function in log scale is shown in~Fig.~\ref{fig:GM_vff}. This signal is known to cause stability issues in STFT phase retrieval~\cite[Sec.~6]{alaifari2021ill}. The measurement is contaminated by a multiplicative noise following the normal distribution $\mathcal{N}(1, 0.05^2)$; see~Fig.~\ref{fig:noise} for an example where we plotted {the noise-free  (left) and noisy signal (middle), as well as the absolute error between the two (right)}. 

We consider the cut-off value $\varepsilon = 0.5$ to define the essential support of the ambiguity function in~\Cref{alg1,alg2}. Since $\varepsilon$ is large, we need to measure windows whose ambiguity functions have broader support in the time-frequency domain to perform the inversion. We first use fractional Fourier transformed versions of the dilated Gauss function $\varphi^{50}$ as the measuring windows. Since the essential support of these ambiguity functions is thin and long, we need to take more windows at different angles to achieve broad coverage of the joint essential support in the time-frequency domain: we use $80$ angles with $\alpha_j = (j-1)\pi/40$, $j=1,\ldots,80$. The recovery is shown in~Fig.~\ref{fig:GM_Gauss}, compared with the case using the standard Gauss as the measuring function. As seen in~Fig.~\ref{fig:GM_vff}, the ambiguity function of the ground truth has almost disjoint support in the time-frequency domain. It thus cannot be stably recovered by a single Gauss window, as shown in~Fig.~\ref{fig:Gauss1_GM}, while the multiple fractional window strategy provides a stable recovery even from noisy measurements; see~Fig.~\ref{fig:Gauss50-80_GM}. Similarly, the reason why Fig.~\ref{fig:GM_vff} only recovers the ground truth over a small portion of the domain is that the ambiguity function associated with the standard Gauss window has a very limited essential support in the time-frequency domain compared to the ambiguity function of the ground truth signal (see Figs.~\ref{fig:frac_gauss_show_1} and~\ref{fig:GM_vff}). While decreasing the cut-off value $\varepsilon$ may improve the reconstruction, it also reduces the stability of the algorithm when using the standard Gauss window.

\begin{figure}
    \centering
        \subfloat[Recovery with $80$ FRFT dilated Gaussian windows]{\includegraphics[width = 0.48\textwidth]{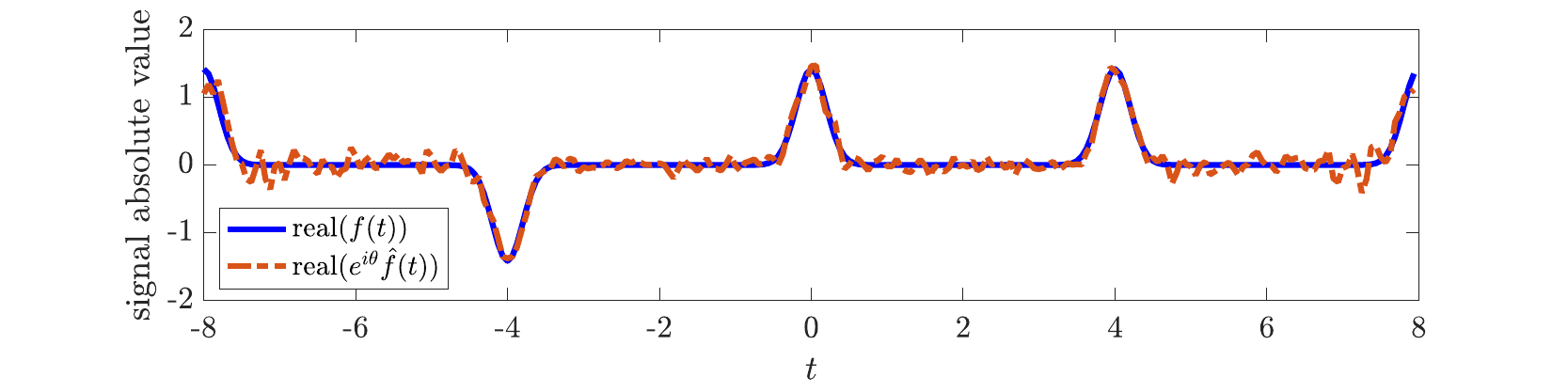}\label{fig:Gauss50-80_GM}}
    \subfloat[Recovery with a standard Gauss  window]{\includegraphics[width = 0.48\textwidth]{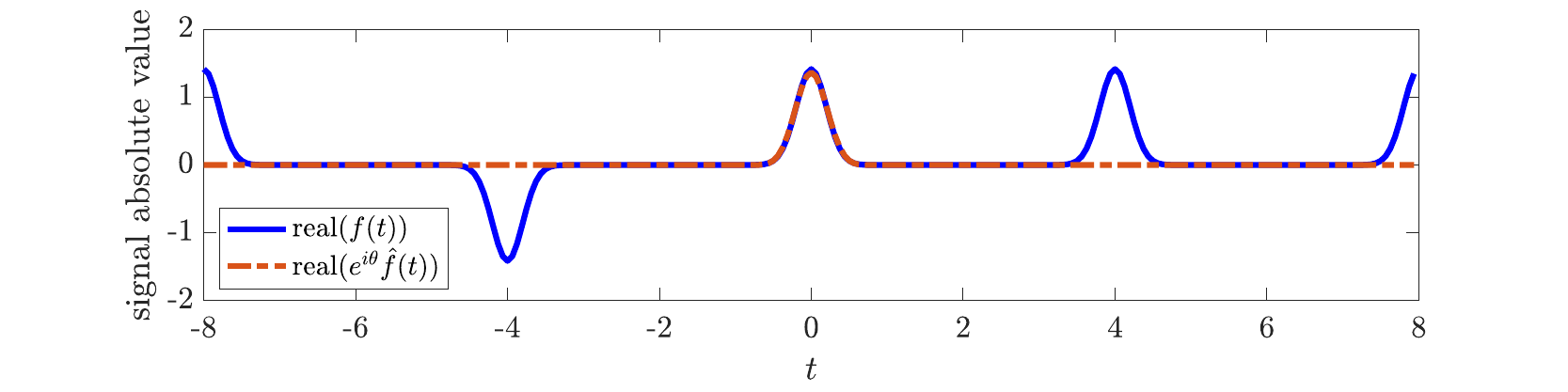}\label{fig:Gauss1_GM}}
    \caption{Multi-modal function example: (a) Recovery using $80$ fractional Fourier transforms of the dilated Gauss window at various angles; (b) Recovery using a single standard Gauss window. The ground truth is plotted in blue. }
    \label{fig:GM_Gauss}
\end{figure}

Next, we also use Hermite functions as the measuring windows. In~Fig.~\ref{fig:Hermite_2}, we use $h_{50}$ and $h_{100}$, the Hermite functions of degree $50$ and $100$, to perform STFT phase retrieval following~\Cref{alg2}. We recover most features of the multi-modal function with minimal impact from the measurement noise illustrated in~Fig.~\ref{fig:noise}. We also compare the inversion result using the single window $h_{100}$ in~Fig.~\ref{fig:Hermite_1}. Although $\mathcal{A} h_{100}$ has extensive overall coverage in the time-frequency domain, it also has $100$ circular rings that are excluded since their absolute value is below the cutoff value $\varepsilon$. As a result, the reconstruction is affected by those missing data. The ambiguity function  $\mathcal{A} h_{50}$ of the additional window  $h_{50}$ makes up most of those ``missing'' rings, so the reconstruction is much better.

The large cut-off value $\varepsilon$ mitigates the impact of noise, while using multiple windows improves the coverage in the time-frequency domain of the joint essential support of the corresponding ambiguity functions. This is evident from both using multiple fractional Fourier transformed dilated Gauss functions and Hermite functions of various degrees.

\begin{figure}
    \centering
    \subfloat[Using $2$ Hermite windows $h_{50}$ and $h_{100}$]{\includegraphics[width = 0.50\textwidth]{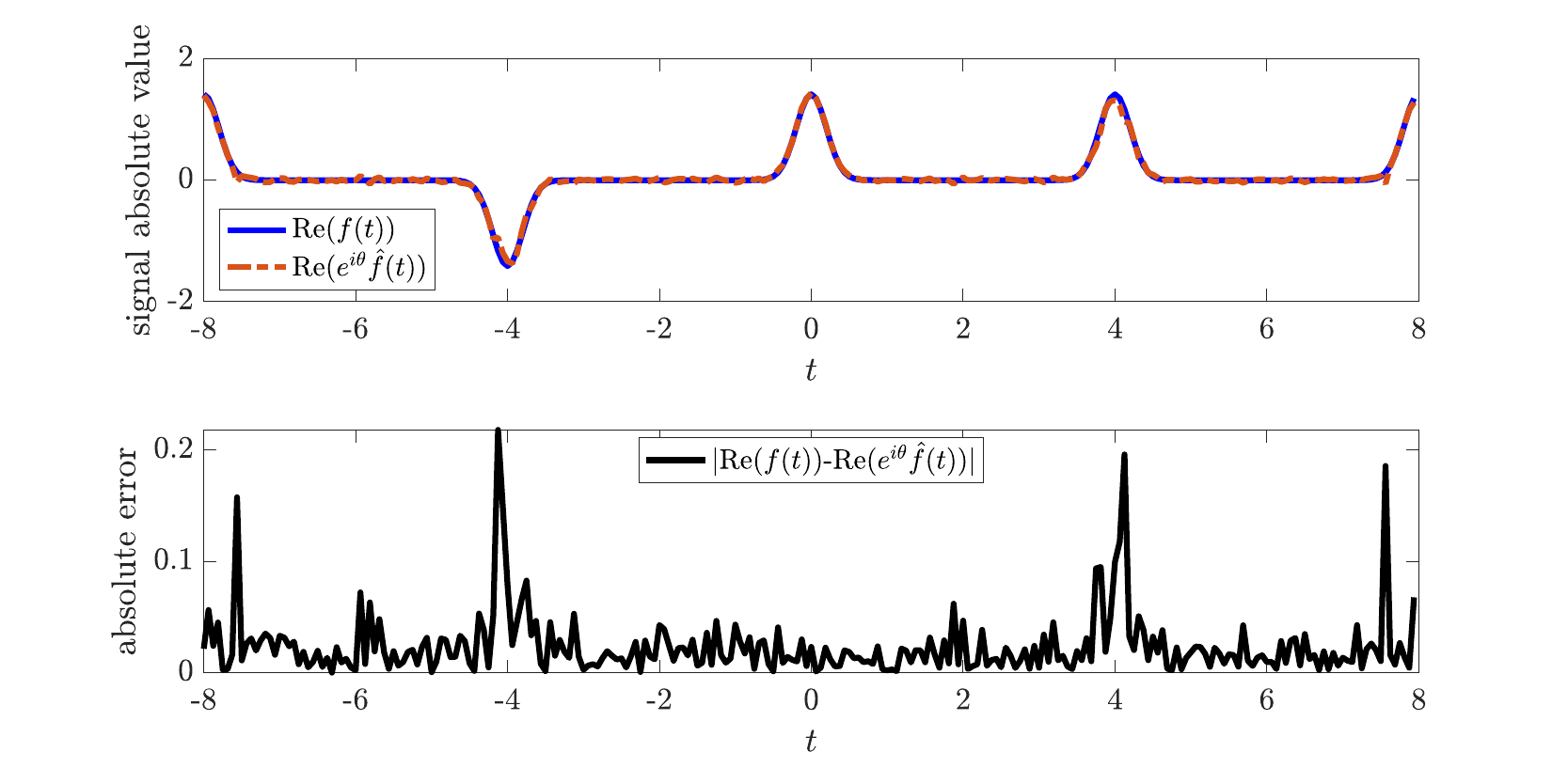}\label{fig:Hermite_2}}
    \subfloat[Using a single Hermite window $h_{100}$]{\includegraphics[width = 0.50\textwidth]{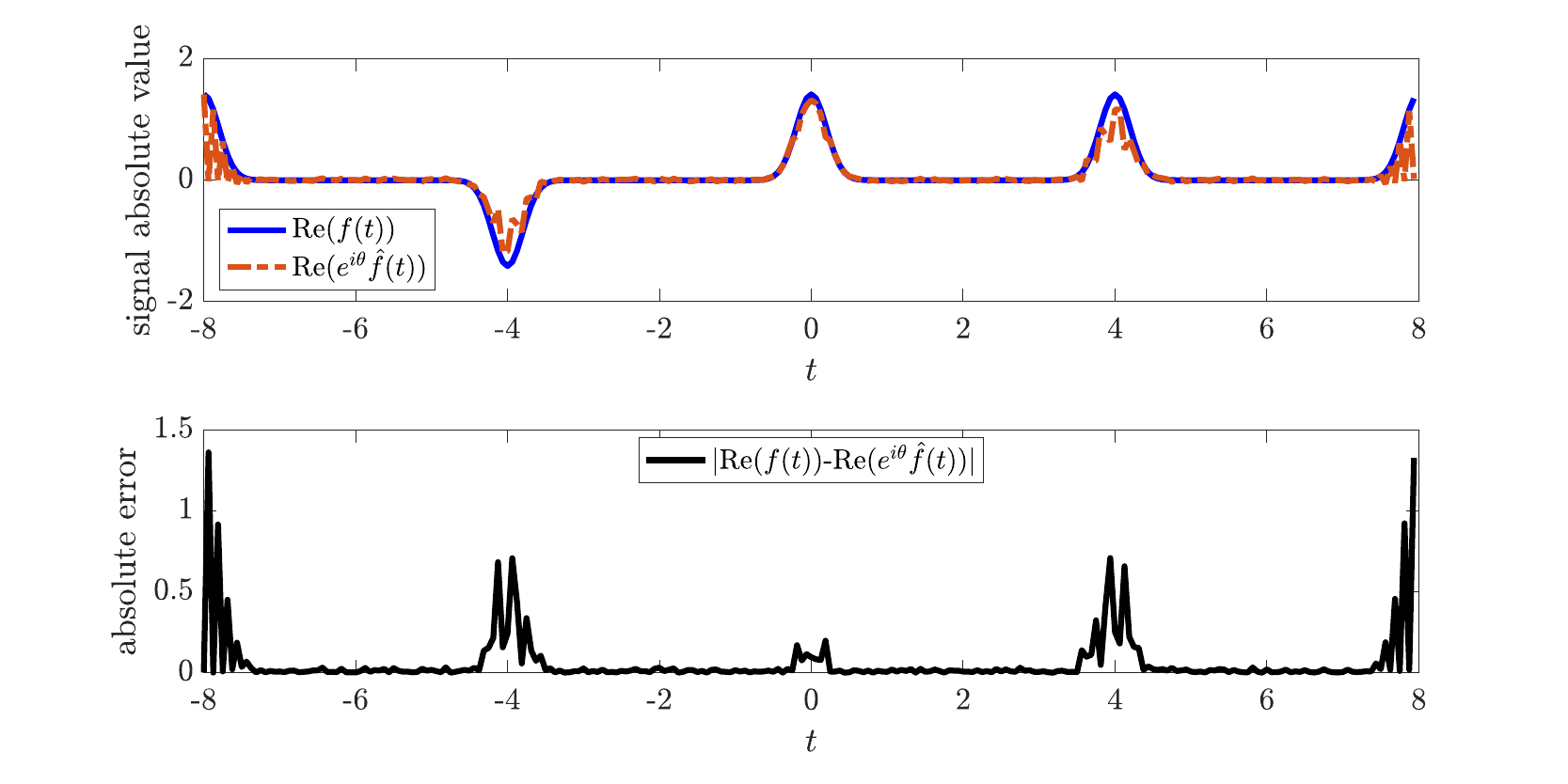}\label{fig:Hermite_1}}
    \caption{Multi-modal function example: (a) Recovery using two Hermite windows $h_{50}$  and $h_{100}$, where the relative error is $0.0881$; (b) Recovery using the single Hermite window $h_{100}$, and the relative error measured by~\eqref{eq:obj} is $0.4417$. The ground truth is shown in blue, the reconstructions obtained using Algorithm~\ref{alg2} are depicted as red dashed curves, and the absolute errors in the real components are illustrated with black solid lines. }
    \label{fig:GM_Hermite}
\end{figure}

\subsection{Audio data example}
Our final numerical example involves recovering a piece of an audio signal, as illustrated in Figure~\ref{fig:audio-full}. This audio signal is significantly more complex than the earlier synthetic examples, yet the proposed multi-window strategy proves effective for the phase retrieval task.

We utilize Algorithm~\ref{alg2} with different sets of Hermite functions. The resulting audio files, saved as \texttt{.wav} files, are provided in the supplementary materials. To visualize the reconstruction performance, we plot a small segment of both the true and recovered signals. The true audio signal is shown in Figure~\ref{fig:true-audio}, while the recovered segments are displayed in Figures~\ref{fig:hermit-1-audio}, \ref{fig:hermit-2-audio}, and \ref{fig:hermit-10-audio}, corresponding to the use of the window function $h_1$, the pair $\{h_1, h_2\}$, and the set $\{h_{0},h_5,h_{10}, h_{15}, h_{20}\}$, respectively. As more windows are incorporated according to Algorithm~\ref{alg2}, the quality of the audio signal recovery progressively improves.

\begin{figure}
\includegraphics[width =1.0\textwidth]{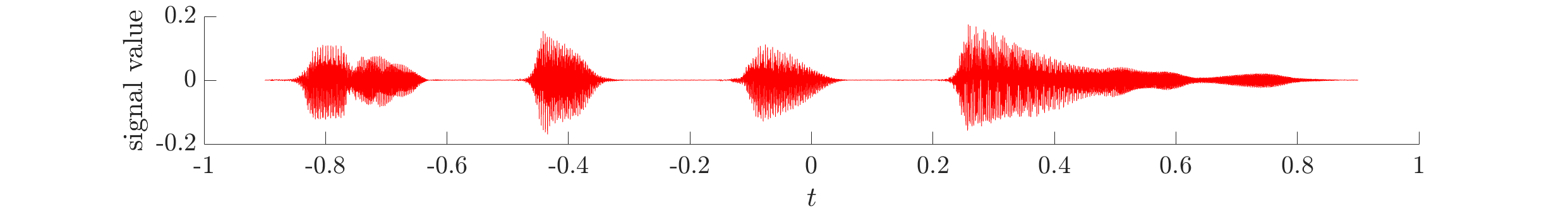}
 \caption{The reference audio signal (in full length).} \label{fig:audio-full}
\end{figure}

\begin{figure}
    \centering
    \subfloat[True signal segment]{\includegraphics[width = 0.5\textwidth]{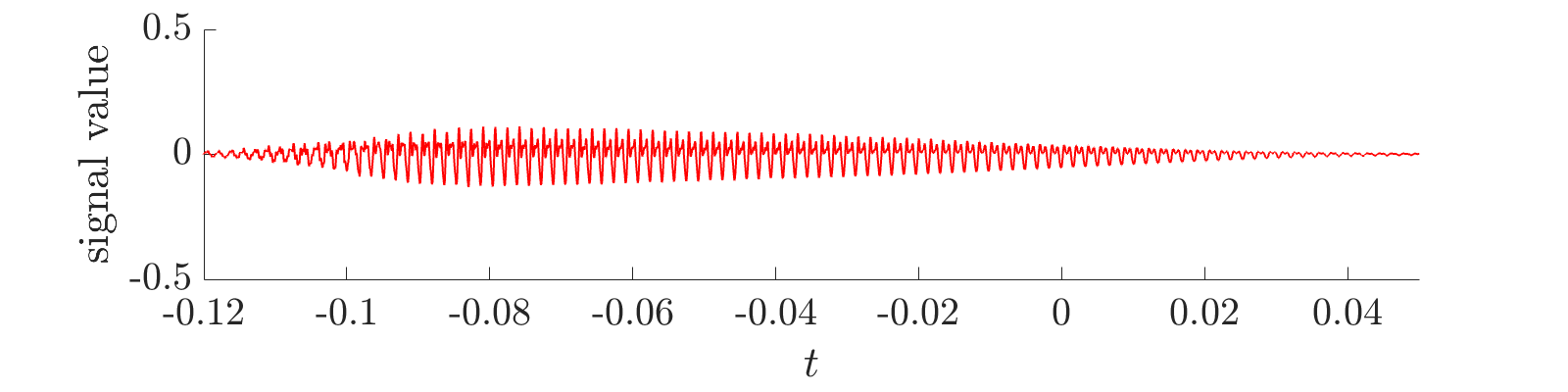}\label{fig:true-audio}}
    \subfloat[Recover with Hermite window $h_1$]{\includegraphics[width = 0.5\textwidth]{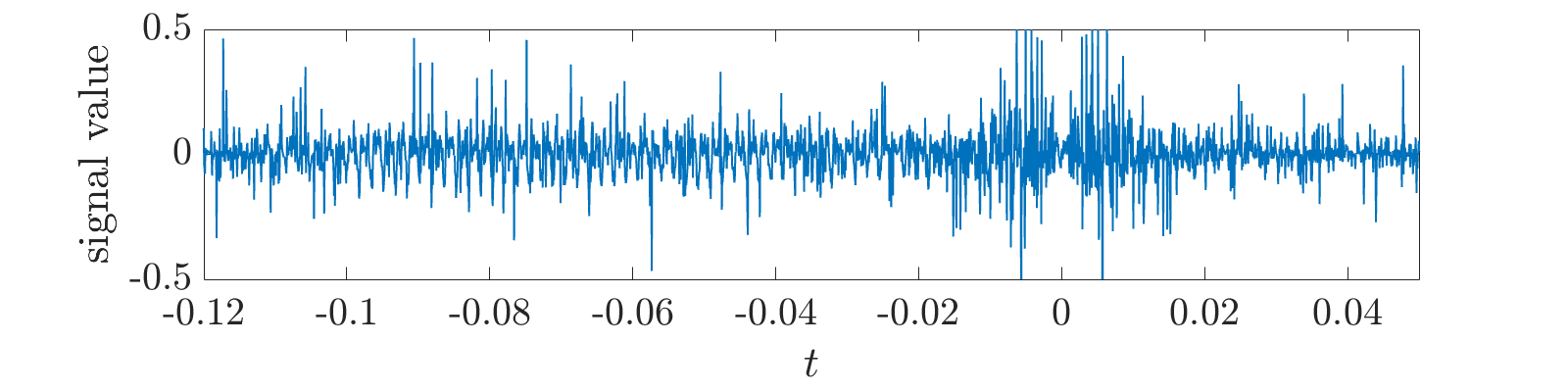}\label{fig:hermit-1-audio}}\\
    \subfloat[Recover with Hermite windows $h_1,h_2$]{\includegraphics[width = 0.5\textwidth]{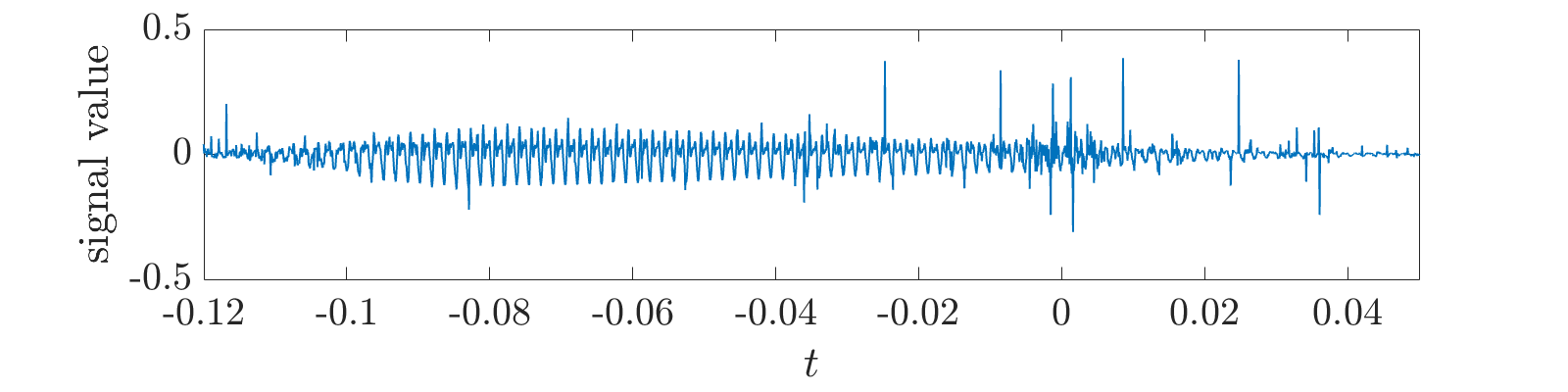}\label{fig:hermit-2-audio}}
    \subfloat[Recover with Hermite windows $\{h_{0},h_5,h_{10}, h_{15}, h_{20}\}$]{\includegraphics[width = 0.5\textwidth]{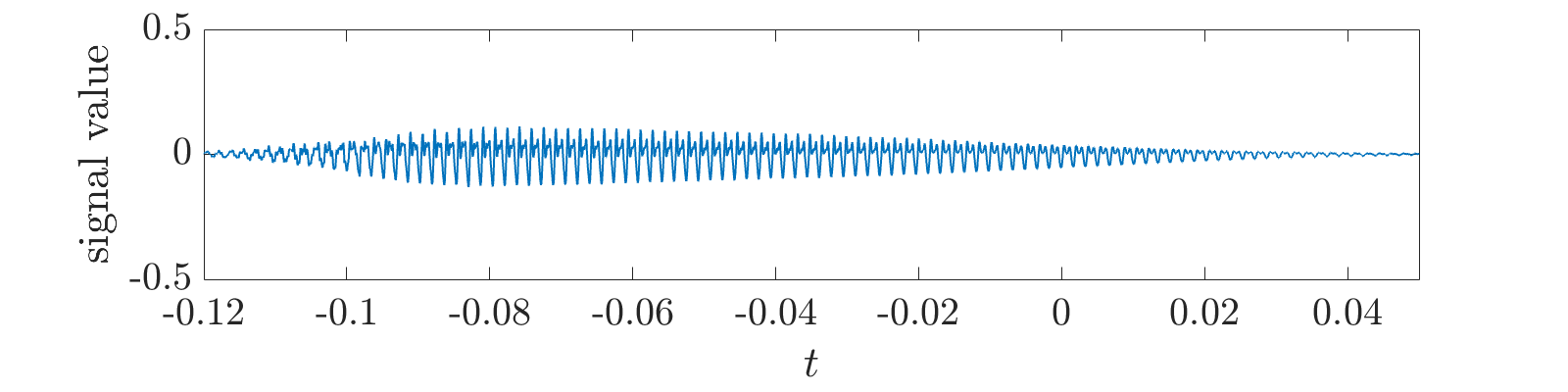}\label{fig:hermit-10-audio}}
    \caption{Audio data STFT phase retrieval example: (a) The true signal segment; (b) Recovery using single Hermite window $h_{1}$; (c) Recovery using two Hermite windows $h_{1}$ and $h_{2}$; (d) Recovery using five Hermite windows $\{h_{0},h_5,h_{10}, h_{15}, h_{20}\}$.}\label{fig:audio}
\end{figure}

\section{Conclusions}\label{sec:conclusions}
In this paper, we introduce two multi-window strategies to improve the stability of STFT phase retrieval. The first approach uses fractional Fourier transforms of dilated Gaussian functions at various angles, while the second employs Hermite functions of different degrees as window functions. As proven in our theory, both methods effectively broaden the coverage of the ambiguity function in the time-frequency domain, significantly mitigating the instability issues associated with traditional single-window approaches. Our numerical experiments demonstrate that these multi-window strategies greatly improve the stability and accuracy of phase retrieval, even for complex audio signals. These results suggest that combining measurements from carefully chosen windows can lead to more robust and reliable signal reconstruction through our simple and direct Algorithms~\ref{alg1} and~\ref{alg2}.

\bibliographystyle{siamplain}
\bibliography{References}
\end{document}